\documentclass[letterpaper, 11pt,  reqno]{amsart}

\usepackage{amsmath,amssymb,amscd,amsthm,amsxtra, esint}

\allowdisplaybreaks[2]

\newtheorem{theorem}{Theorem} [section]

\newtheorem{lemma}[theorem]{Lemma}
\newtheorem{proposition}[theorem]{Proposition}
\newtheorem{remark}[theorem]{Remark}

\DeclareMathOperator*{\supp}{supp}

\newcommand{\I}{\hspace{0.5mm}\text{I}\hspace{0.5mm}}
\newcommand{\II}{\text{I \hspace{-2.8mm} I} }

\newcommand{\noi}{\noindent}
\newcommand{\Z}{\mathbb{Z}}
\newcommand{\R}{\mathbb{R}}
\newcommand{\C}{\mathbb{C}}
\newcommand{\T}{\mathbb{T}}

\let\Re=\undefined\DeclareMathOperator*{\Re}{Re}
\let\Im=\undefined\DeclareMathOperator*{\Im}{Im}

\let\P= \undefined
\newcommand{\P}{\mathbf{P}}

\newcommand{\E}{\mathbb{E}}

\newcommand{\N}{\mathcal{N}}

\newcommand{\al}{\alpha}

\newcommand{\dl}{\delta}

\newcommand{\nb}{\nabla}

\newcommand{\Dl}{\Delta}
\newcommand{\eps}{\varepsilon}

\newcommand{\g}{\gamma}

\newcommand{\ld}{\lambda}

\newcommand{\Si}{\Sigma}
\newcommand{\ft}{\widehat}

\newcommand{\wt}{\widetilde}
\newcommand{\cj}{\overline}

\newcommand{\dt}{\partial_t}

\renewcommand{\l}{\ell}
\renewcommand{\o}{\omega}
\renewcommand{\O}{\Omega}

\newcommand{\les}{\lesssim}
\newcommand{\ges}{\gtrsim}

\newcommand{\jb}[1]
{\langle #1 \rangle}

\newcommand{\pa}{\partial}
\newcommand{\pP}{\mathbf{P}}
\newcommand{\M}{\mathcal{M}}

\newtheorem*{ackno}{Acknowledgement}

\numberwithin{equation}{section}
\numberwithin{theorem}{section}

\usepackage[english, french]{babel}

\begin{document}
\baselineskip = 14pt

\selectlanguage{english}

\title[Probabilistic GWP of the energy-critical NLW on $\R^3$]
{Probabilistic global well-posedness of the energy-critical
defocusing quintic nonlinear wave equation on  $\R^3$
}

\author[T.~Oh and O.~Pocovnicu]
{Tadahiro Oh and   Oana Pocovnicu}

\address{
Tadahiro Oh\\
School of Mathematics\\
The University of Edinburgh\\
and The Maxwell Institute for the Mathematical Sciences\\
James Clerk Maxwell Building\\
The King's Buildings\\
Peter Guthrie Tait Road\\
Edinburgh\\ 
EH9 3FD\\
 United Kingdom}
 
\email{hiro.oh@ed.ac.uk}

\address{
Oana Pocovnicu\\
Department of Mathematics\\
Princeton University\\
Fine Hall\\
Washington Rd.\\
Princeton\\ NJ 08544-1000\\
USA\\
and 
Department of Mathematics\\
Heriot-Watt University
and The Maxwell Institute for the Mathematical Sciences\\
Edinburgh\\
EH14 4AS\\United Kingdom}

\email{o.pocovnicu@hw.ac.uk}


%
%


\maketitle

\vspace{-10mm}

\begin{abstract}
We prove almost sure global well-posedness of the energy-critical defocusing quintic nonlinear wave equation on $\R^3$
with random initial data in $ H^s(\R^3) \times H^{s-1}(\R^3)$
for $s > \frac 12$.
The main new ingredient is a uniform probabilistic energy bound
for approximating random solutions.
\end{abstract}

\begin{otherlanguage}{french}
\begin{abstract}

On consid\`ere l'\'equation des ondes critique d\'efocalisante dans $\R^3$ \`a don\'ees initiales al\'eatoires
dans $H^s(\R^3) \times H^{s-1}(\R^3)$, avec $s > \frac 12$. 
On \'etablit que ce probl\`eme est globalement bien-pos\'e presque s\^urement. 
Le principal ingredient nouveau de la preuve
est une estimation probabiliste uniforme de l'\'energie des solutions approch\'ees.

\end{abstract}
\end{otherlanguage}


%
%

\baselineskip = 15pt

\section{Introduction}

\subsection{Nonlinear wave equation}

We consider the Cauchy problem for the energy-critical defocusing quintic nonlinear wave equation (NLW) on $\R^3$:
\begin{equation}\label{NLW}
\begin{cases}
\pa_{t}^2 u-\Delta u+u^5=0 
\\
(u,   \pa_t u)|_{t = 0} = (u_0, u_1), 
\end{cases}
\quad \quad (t,x)\in\R\times\R^3, 
\end{equation} 

\noi
where 
$u$ is a real-valued function.
NLW has been studied extensively
from both  applied and theoretical points of view,
in particular in three spatial dimensions
due to its physical importance.
In this paper, we study 
 the global-in-time behavior of solutions to \eqref{NLW} with {\it random} and {\it rough} initial data
 below the energy space.

It is well known that  the quintic NLW  \eqref{NLW} on $\R^3$ is invariant under the following dilation symmetry:
\begin{align}
u(t,x)\mapsto u_{\lambda}(t,x):=\lambda^{\frac{1}{2}} u(\lambda t, \lambda x).
\label{Zscaling}
\end{align}

\noi
Namely, if $u$ is a solution to \eqref{NLW},
then $u_{\lambda}$ is also a solution to \eqref{NLW} with rescaled initial data.
Recall that  the $\dot{H}^1(\R^3)\times L^2(\R^3)$-norm is invariant under this dilation symmetry:
\begin{equation*}
\|(u_\lambda (0), \pa_t u_\lambda(0))
\|_{\dot{H}^1(\R^3)\times L^2(\R^3)}=
\|(u(0),\pa_t u(0))
\|_{\dot{H}^1(\R^3)\times L^2(\R^3)}.
\end{equation*}

\noi
Moreover,
the   conserved energy $E(u)$ defined by 
\begin{equation}
E(u) = E(u, \pa_t u) := \int_{\R^3}\frac 12(\pa_t u)^2+\frac 12|\nabla u|^2+\frac16
u^6 dx
\label{Zenergy}
\end{equation} 

\noi
is also invariant under the dilation symmetry \eqref{Zscaling}.
This explains why the quintic NLW on $\R^3$
is called  {\it energy-critical}.
In view of Sobolev's inequality: 
$\dot{H}^1(\R^3) \subset L^{6}(\R^3)$, 
we see that 
$E(u, \dt u) < \infty$
if and only if 
\[(u, \dt u)\in \dot{H}^1(\R^3)\times L^2(\R^3).\]

\noi
In the following, we refer to $\dot{H}^1(\R^3)\times L^2(\R^3)$ as the energy space.

Let us briefly recall the known results on global well-posedness of 
the defocusing NLW in the energy space.
For an energy-subcritical defocusing  NLW on $\R^3$ with nonlinearity $|u|^{p-1}u$, $p < 5$, 
the conservation of the energy 
allows us to iterate the local-in-time argument
and obtain 
global well-posedness in $\dot H^1(\R^3)\times L^2(\R^3)$.
The energy-critical defocusing quintic NLW \eqref{NLW} on $\R^3$, however, 
lies at a 
rather 
delicate balance of dispersion by the linear evolution
and concentration due to the nonlinearity, 
and the issue of global well-posedness for \eqref{NLW} is more intricate.
After substantial efforts made by many mathematicians, 
it is now known that 
\eqref{NLW} is globally well-posed in the energy space
and all finite energy solutions scatter
 \cite{ Struwe, Grillakis90, Grillakis92, Shatah_Struwe93, 
Shatah_Struwe, Kapitanski, 
Ginibre, Bahouri_Shatah,
Bahouri_Gerard, 
Tao}.
Lastly, recall that 
the energy-critical quintic NLW
\eqref{NLW} on $\R^3$ is known to be ill-posed
below the energy space
\cite{Christ_Colliander_Tao_main}.

Recently, there has been a significant development
in probabilistic construction of local-in-time and global-in-time solutions to 
hyperbolic and dispersive PDEs 
below certain regularity thresholds (such as a scaling critical regularity), 
where the equations are known to be ill-posed deterministically.
In particular, 
following the methodology developed in \cite{Bourgain96, BTI, BOP1, Poc}, 
one can easily prove almost sure local well-posedness of \eqref{NLW} 
below the energy space (Theorem \ref{THM:LWP}).
Therefore, it is natural to study the long time behavior
of such local solutions constructed in a probabilistic manner.

%
%


Our main goal in this paper is to 
prove almost sure
global well-posedness  of 
\eqref{NLW} below the energy space
under suitable randomization of  initial data.
See Theorem \ref{THM:GWP} below.
In particular, this settles 
the question
of almost sure global well-posedness 
for {\it large},  random,  and rough initial data,
in the physically important case of the energy-critical NLW on $\R^3$.
This case was not addressed in the previous works 
on the subject. 
Indeed, previously, 
L\"uhrmann-Mendelson \cite{LM}
proved almost sure global well-posedness
below the scaling critical regularity for energy-subcritical (sub-quintic)  
NLW on $\R^3$.
In the same paper, they also proved
almost sure small data global well-posedness 
for the energy-critical NLW \eqref{NLW} on $\R^3$.
We point out that 
the methods used in \cite{LM} 
are
specific to the energy-subcritical or small data setting
and are not applicable to our problem.
In fact, a new method was needed for
studying the global behavior of solutions to an  energy-critical equation with large random initial data.
Recently, 
the second author \cite{Poc}
successfully implemented such a method
and 
proved almost sure global well-posedness 
below the energy space
for the energy-critical NLW
on $\R^d$, $d=4,5$, 
with large random initial data.
The argument in \cite{Poc}, however,
fails in  the case of $d=3$,
and thus we need to develop
additional new ideas and perform a more intricate analysis
to treat \eqref{NLW} on $\R^3$.

\subsection{Wiener randomization}
In this subsection, we discuss the randomization for functions on $\R^3$
that we employ for our main result.

Following the works of Bourgain \cite{Bourgain96} and Burq-Tzvetkov \cite{BTI}, 
there have been 
many  results on 
probabilistic construction of  solutions to evolution equations 
via randomization of initial data.
On a compact manifold $M$, 
 there is a countable (orthonormal) basis $\{e_n\}_{n\in\mathbb N}$ of $L^2(M)$ consisting of
eigenfunctions of the Laplace-Beltrami operator.
This gives a natural way to introduce a randomization as follows.
Given $u_0=\sum_{n=1}^\infty \ft {u}_n e_n\in H^s(M)$,
we can  define its randomization $u_0^\o$ by
\begin{align}
u_0^\omega :=  \sum_{n=1}^\infty g_n(\omega) \ft u_n e_n,
\label{Ziv}
\end{align}

\noi
where $\{g_n\}_{n\in\mathbb N}$
is a sequence of independent mean zero random variables, satisfying certain moment estimates.
When $M = \T^d$, we can express $u_0^\o$ in \eqref{Ziv} as 
\begin{align}
u_0^\omega
= \Xi(\o) * u_0, 
\label{Ziv0}
\end{align}

\noi
where  $\Xi$ is a random distribution given by\footnote{On $\T^d$, 
we have $e^{2\pi i n\cdot x}  = \phi (D - n) \dl$, $n \in \Z^d$, where $\phi = \chi_{B(0, \frac{1}{2})}$.
Then, $\Xi$ in \eqref{Ziv1a} can be written as 
\begin{align*}
 \Xi(\o) = \sum_{n\in\Z^d} g_n(\omega) \phi(D-n) \dl.
\end{align*}

\noi
Compare this with \eqref{Ziv1b} below.
}  
\begin{align}
 \Xi(\o) = \sum_{n\in\Z^d} g_n(\omega) e^{2\pi i n \cdot x}.
 \label{Ziv1a}
\end{align}

\noi
In particular, if $\{ g_n\}_{n \in \mathbb Z^d}$ is a sequence of 
independent standard Gaussian random variables, 
then  $\Xi$ in \eqref{Ziv1a} corresponds to the (mean zero Gaussian) white noise on $\T^d$.
In this case, we can call the randomization $u_0^\o$ given by \eqref{Ziv} and \eqref{Ziv0}
the white noise randomization of $u_0$.
See Remark \ref{REM:white} below.

On the Euclidean space $\R^d$, however,
 there is no countable
basis of $L^2(\R^d)$ consisting of eigenfunctions of the Laplacian
and thus there is no `natural' way to introduce a randomization 
of functions
as in \eqref{Ziv}.
Randomizations for functions on $\R^d$ have been considered
 with respect to some other countable bases of $L^2(\R^d)$
such as a countable basis of the eigenfunctions of the Laplacian with a confining potential, 
for example,  the harmonic oscillator
 $-\Delta + |x|^2$,   \cite{Thomann, BTT}.
In the following, however, 
we consider a simple randomization for functions on $\R^d$,
naturally associated to the Wiener decomposition of the frequency space $\R^d_\xi$.
See also \cite{LM, BOP1, BOP2}.

Let $Q_n$ be  the unit cube $Q_n:=n+\big[-\frac 12, \frac 12 \big)^d$ centered at $n\in \Z^d$.
For simplicity, we set $Q := Q_0$.
The Wiener decomposition \cite{W}
of the frequency space $\R^d_\xi$
is given by 
 the uniform partition:
 $\R^d = \bigcup_{n \in \Z^d} Q_n$.
Clearly, given a function $u$ on $\R^d$, we have 
\begin{equation}
 u = \sum_{n \in \Z^d} \chi_{Q_n} (D) u 
=  \sum_{n \in \Z^d} \chi_{Q} (D - n) u. 
\label{Wiener}
\end{equation}

\noi
Here, $ \chi_{Q_n} (D) $ denotes the Fourier multiplier operator
with symbol $\chi_{Q_n}$.

Next, we consider the smoothed version of the decomposition \eqref{Wiener}.
Let $\psi \in \mathcal{S}(\R^d)$ 
be such that $\supp \psi \subset [-1, 1]^d$,  $\psi(-\xi ) = \overline{\psi(\xi)}$,
and 
\[ \sum_{n \in \Z^d} \psi(\xi - n) \equiv 1 \quad \text{for all }\xi \in \R^d.\]

\noi
Then, any function $u$ on $\R^d$ can be written as
\begin{equation}
 u = \sum_{n \in \Z^d} \psi(D-n) u,
\label{Ziv2}
 \end{equation}

\noi
where $ \psi (D-n) $ denotes the Fourier multiplier operator
with symbol $\psi (\,\cdot\, -n)$.

We now introduce  a randomization 
adapted to the uniform decomposition \eqref{Ziv2}.
For $j = 0, 1$, 
let $\{g_{n,j}\}_{n \in \Z^d}$ be a sequence of mean zero complex-valued random variables
on a probability space $(\Omega, \mathcal{F}, P)$
such that $g_{-n,j}=\overline{g_{n,j}}$
for all $n\in\Z^d$, $j=0,1$.
In particular,  $g_{0, j}$ is real-valued.
Moreover, we  assume that 
$\{g_{0,j}, \Re g_{n,j}, \Im g_{n,j}\}_{n\in\mathcal I, j=0,1}$ are independent,
where the index set $\mathcal{I}$ is defined by 
\begin{equation}
\mathcal I:=\bigcup_{k=0}^{d-1} \Z^k\times \Z_{+}\times \{0\}^{d-k-1}.
\label{Index}
\end{equation}

\noi
Note that $\Z^d = \mathcal I \cup (-\mathcal I)\cup \{0\}$.
Then, given a pair $(u_0, u_1)$ of  functions on $\R^d$, 
we  define the \emph{Wiener randomization} $(u_0^\omega, u_1^\omega)$
of $(u_0,u_1)$ by
\begin{align}
(u_0^\omega, u_1^\omega) : & = 
(\Xi_0 (\o) * u_0,\,  \Xi_1(\o)*u_1) \notag \\
& = 
\bigg(\sum_{n \in \Z^d} g_{n,0} (\omega) \psi(D-n) u_0,
\sum_{n \in \Z^d} g_{n,1} (\omega) \psi(D-n) u_1\bigg).
\label{R1}
\end{align}

\noi
Here, $\Xi_0$ and $\Xi_1$ are random distributions given 
by 
\begin{align}
 \Xi_j(\o) = \sum_{n\in \Z^d} g_{n, j} (\omega) \psi (D-n) \dl, \ \ j = 0, 1, 
 \label{Ziv1b}
\end{align}

\noi
where $\dl$ denotes the Dirac delta distribution.
Note that, if $u_0$ and $u_1$ are real-valued,
then their randomizations $u_0^\omega$ and $u_1^\omega$ defined in \eqref{R1}
are also real-valued.

We make the following assumption on the 
 probability distributions
$\mu_{n,j}$ for $g_{n, j}$;
 there exists $c>0$ such that
\begin{equation}
\int e^{\gamma \cdot x}d\mu_{n,j}(x)\leq e^{c|\gamma|^2}, \quad j = 0, 1, 
\label{cond}
\end{equation}
	
\noindent
for all $n \in \Z^d$,  
(i) all $\gamma \in \R$ when $n = 0$,
and (ii)  all $\g \in \R^2$ when $n \in \Z^d \setminus \{0\}$.
Note that \eqref{cond} is satisfied by
standard complex-valued Gaussian random variables,
standard Bernoulli random variables,
and any random variables with compactly supported distributions.

It is easy to see that, if $(u_0,u_1) \in H^s(\R^d)\times H^{s-1}(\R^d)$
for some $s \in \R$,
then  the Wiener randomization $(u_0^\omega, u_1^\omega)$ is
almost surely  in $H^s(\R^d)\times H^{s-1}(\R^d)$. 
Note that, under some non-degeneracy condition on the random variables $\{g_{n, j}\}$, 
 there is almost surely no gain from randomization
in terms of differentiability (see, for example, Lemma B.1 in \cite{BTI}).
Instead, the main feature of 
the Wiener randomization 
 \eqref{R1}
is that $(u_0^\omega, u_1^\omega)$  behaves better in terms of integrability.
More precisely, if $u_j \in L^2(\R^d)$, $j=0,1$,
then  the randomized function $u_j^\omega$ is
almost surely  in $L^p(\R^d)$ for any finite $p \geq 2$.
See \cite{BOP1}.
It is this improved integrability that allows us to construct global solutions to \eqref{NLW} 
below the energy space in a probabilistic manner.

\begin{remark}\rm 
The uniform decomposition \eqref{Ziv2}
comes from the modulation symmetry (of $L^2(\R^d)$), i.e.~the translation symmetry on the Fourier side.
As such, the uniform decomposition \eqref{Ziv2}
and the Wiener randomization \eqref{R1}
are closely related to the modulation spaces.
See \cite{BOP1} for  more  discussion on this issue.

\end{remark}

\begin{remark}\label{REM:white}\rm
Let  $M = \T^d$.
In this case,  if $\{ g_n\}_{n \in \mathbb \Z^d}$ is a sequence of 
independent standard Gaussian random variables, 
then 
  $\Xi$ in \eqref{Ziv1a} represents the  white noise on $\T^d$
  and the randomization  
  \eqref{Ziv0}
  gives the white noise randomization
  for functions defined on $\T^d$.
Given $L > 0$, 
let  $\Xi_L$ be the white noise on 
$\T_L ^d: = (\R / L\Z)^d \simeq[-\frac{L}{2}, \frac L2)^d$
defined by 
 \[ \Xi_L(x; \o) = \sum_{n \in \mathbb{Z}^d} \frac{g_n(\o)}{L^\frac{d}{2}} e^{2\pi i \frac {n }{L}\cdot x}.\]

\noi
Then, we can also consider the white noise randomization 
$u_0^\omega
= \Xi_L * u_0$ for functions on $\T_L^d$.
One of the main features of this randomization is 
the gain in integrability;
if $u_0 \in L^2(\T_L^d)$, 
then  the randomized function $u_0^\omega$ is
almost surely  in $L^p(\T_L^d)$ for any finite $p \geq 2$.
As mentioned above, 
 this improved integrability also holds 
 for 
the Wiener randomization \eqref{R1}
for functions on $\R^d$.

Given  a function $u_0$ on $\R^d$,
 one may be tempted to consider an analogous white noise randomization $u_0^\o := \Xi_{\R^d} * u_0$
 on $\R^d$,
 where $\Xi_{\R^d}$ is the white noise on $\R^d$
 obtained as the limit\footnote{As $L \to \infty$, 
 $\Xi_L$ converges in distribution to the white noise $\Xi_{\R^d}$ 
on $C^{0, s}_\text{loc} (\R^d; \C)$, $s = -\frac{d}{2}-$, viewed as a Fr\'echet space 
endowed with the metric:
\[ d(f, g) = \sum_{k = 1}^\infty \frac{1}{2^k}
\frac{\|f-g\|_{ C^{0, s}([-k, k])}}{1+\|f-g\|_{ C^{0, s}([-k, k])}}.\]

\noi
This can be seen from the corresponding convergence (in distribution) of the 
periodized Brownian motion on $\T^d$ (represented by the Fourier-Wiener series) to the Brownian motion on $\R^d$
in $C^{0, s}_\text{loc} (\R^d; \C)$ with  $s = 1-\frac{d}{2}-$.
}  of $\Xi_L$ as $L \to \infty$.
 Such a randomization, however, 
is not suitable for our problem due to the lack of (global) integrability. 
 For example, given $u_0 \in L^2(\R^d)$,
it follows from $\mathbb{E} \big[ \Xi_{\R^d}(x) \cj{\Xi_{\R^d}(y)} \big] = \dl(x - y)$
that 
 \begin{align*}
 \mathbb{E}\big[ \| \Xi_{\R^d} * u_0\|_{L^2(\R^d)}^2\big]
& = \int \mathbb{E} \bigg[\int \Xi_{\R^d}(x-y) u_0(y) dy 
\cj {\int \Xi_{\R^d}(x-z) u_0(z) dz } \bigg] \, dx\\
& = \int \| u_0\|_{L^2(\R^d)}^2  dx
= \infty.
 \end{align*}

This shows that while the white noise randomization is useful in studying evolution equations on $\T^d$,
it is not suitable  on $\R^d$, at least for our problem.
The Wiener randomization discussed above can be regarded
as a suitable adaptation of the white noise randomization on $\R^d$,
but on a fixed scale.
See \cite{BOP2} for the effect of the Wiener randomization
based on dilated cubes.
\end{remark}

\subsection{Main result}
Our main goal in this paper
is to prove almost sure global well-posedness of \eqref{NLW} on $\R^3$
below the energy space (Theorem \ref{THM:GWP}).

We use the following shorthand notations
for products of  Sobolev spaces:
\[\mathcal{H}^s(\R^3) : = H^s(\R^3)\times H^{s-1}(\R^3)
\quad \text{and}\quad
\dot{\mathcal{H}}^s(\R^3) : = \dot{H}^s(\R^3)\times \dot{H}^{s-1}(\R^3).
\]

\noi
We also denote by $S(t)$ the propagator for the linear wave equation
given by 
\begin{equation}\label{Zlinear}
S(t)\left(u_0, u_1\right):=\cos(t|\nabla|)u_0+\frac{\sin (t|\nabla|)}{|\nabla|}u_1.
\end{equation}

\noi
We  first present  the following result on almost sure local well-posedness of \eqref{NLW} below the energy space.

\begin{theorem}[Almost sure local well-posedness]
\label{THM:LWP}
Let $ s \in [0, 1)$. 
Given  $(u_0, u_1) \in \mathcal{H}^s(\R^3)$, 
let $(u_0^\omega, u_1^\omega)$
be the Wiener randomization defined in \eqref{R1},
satisfying \eqref{cond}. 
Then, the energy-critical defocusing quintic NLW \eqref{NLW}
on $\R^3$ is almost surely locally well-posed
with  respect to the Wiener randomization $(u_0^\omega,u_1^\omega)$
as initial data.
More precisely, there exist $C,c,\gamma>0$ such that for each $T\ll1$,
there exists a set $\Omega_T\subset \Omega$ with the following properties:

\medskip
\noindent
{\rm(i)} $P(\Omega_T^c)<C\exp(-\frac{c}{T^\gamma})$.

\medskip
\noindent
{\rm(ii)} For each $\omega\in \Omega_T$, there exists a unique solution $u^\omega$ to \eqref{NLW}
with $(u^\omega, \pa_t u^\omega)|_{t = 0}
=(u_0^\omega, u_1^\omega)$ in the class:
\begin{align*}
\big(S(t)(u_0^\omega, u_1^\omega), \dt S(t)(u_0^\omega, u_1^\omega)\big)
+C([-T,T]; \mathcal{H}^1(\R^3))
\subset C([-T,T]; \mathcal{H}^s(\R^3)).
\end{align*}

\noi
Here,  uniqueness  holds
in a ball centered at 
$S(\cdot)(u_0^\omega, u_1^\omega)$ 
in 
\[C([-T,T]; \dot{H}^1(\R^3))\cap 
L^{5}([-T,T]; L^{10}(\R^3)).\]

\end{theorem}

This theorem  is in the spirit of 
the almost sure local well-posedness results in \cite{BTI, BOP1, Poc}.
Namely, given random initial data $(u_0^\o, u_1^\o)$, 
denote the linear
 and nonlinear parts of the solution 
$u^\omega$
to  \eqref{NLW}
by 
\begin{equation}
z^{\omega}(t): =S(t)(u_0^{\omega}, u_1^{\omega})
\qquad \text{and}
\qquad v^\omega:=u^\omega-z^\omega.
\label{z}
\end{equation}

\noi
Then,  \eqref{NLW} can be reformulated  as the following perturbed
NLW: 
\begin{equation}\label{v}
\begin{cases}
\pa_{t}^2 v^\omega-\Delta v^\omega+(v^\omega+z^\omega)^5=0\\
(v^\omega,\pa_t v^\omega)|_{t=0}=(0,0).
\end{cases}
\end{equation}

\noi
In view of the usual deterministic Strichartz estimates (Lemma \ref{LEM:Strichartz})
and the probabilistic Strichartz estimates (Lemma \ref{LEM:Str}), 
a simple fixed point argument allows us to 
construct  a solution $v^\o$ to \eqref{v}
in $C([-T, T];H^1(\R^3))$ for each $\o$ belonging to some appropriate set $ \O_T$.
This yields Theorem \ref{THM:LWP}. 
As this argument is standard, we omit the proof of Theorem \ref{THM:LWP}.
See \cite{Poc} for details.

\begin{remark}\label{REM:LWP}\rm
(i) Note that the regularity $s = 0$ is the lowest regularity 
for which one can prove
almost sure local well-posedness by this argument of 
 constructing a solution $v^\o$ to \eqref{v} in $C([-T, T]; H^1(\R^3))$.  
This is due to the fact that the nonlinear Duhamel term in 
\eqref{Duhamel}  gains exactly one derivative.

\smallskip

\noi
(ii) In the definition of the Wiener randomization \eqref{R1},
we used a smooth cutoff function $\psi$.
Theorem \ref{THM:LWP} still holds
even when we replace $\psi$ by the sharp characteristic function 
$\chi_{Q}$ of the unit  cube $Q$.
The same comment holds
for Theorem \ref{THM:GWP}.
See also Remark \ref{REM:GWP} (iii) below.

\end{remark}

Next, we turn our attention to the global-in-time behavior
of solutions with random initial data below the energy space.
The following is the main result of this paper.
%

\begin{theorem}[Almost sure global well-posedness]
\label{THM:GWP}
Let $s \in (\frac 12, 1)$.
Given  $(u_0, u_1) \in \mathcal{H}^s(\R^3)$, 
let $(u_0^\omega, u_1^\omega)$
be the Wiener randomization defined in \eqref{R1},
satisfying \eqref{cond}. 
Then, the energy-critical defocusing quintic NLW \eqref{NLW}
on $\R^3$ is almost surely globally well-posed
with  respect to the Wiener randomization $(u_0^\omega,u_1^\omega)$
as initial data.
More precisely,
there exists a set $ \O_{(u_0, u_1)}\subset \Omega $ of probability 1
such that, 
for  every $\o \in \O_{(u_0, u_1)}$, there exists a unique solution $u$ to \eqref{NLW}
in the class:
\begin{align*}
\big(S(t)(u_0^\omega, u_1^\omega), \dt S(t)(u_0^\omega, u_1^\omega)\big)
+C(\R; \mathcal{H}^1(\R^3))
\subset C(\R; \mathcal{H}^s(\R^3)).
\end{align*}

\end{theorem}

Before explaining the main ideas of the 
proof of Theorem \ref{THM:GWP},
we first discuss the previous results directly relevant to our problem.
Interested readers are referred to \cite{Poc} for a thorough list of references
on almost sure global well-posedness of evolution equations
with random initial data.


Previously, L\"uhrmann-Mendelson \cite{LM} considered energy-subcritical defocusing NLW
on $\R^3$ with nonlinearity $|u|^{p-1}u$, $ p < 5$,  
with random initial data of the form \eqref{R1}.
In particular, for $\frac 14 (7 + \sqrt{73})\simeq 3.89 < p < 5$,  
they proved almost sure global well-posedness below
the scaling critical Sobolev regularity $s_\text{crit} := \frac 32 - \frac 2 {p-1}<1$.
Their approach is based on the probabilistic high-low method
introduced by Colliander-Oh \cite{Colliand_Oh} in the study
of  the cubic nonlinear Schr\"odinger equation (NLS) on $\T$
with random initial data.
This method is  an adaptation of Bourgain's high-low method \cite{Bourgain98}
to the probabilistic setting
and is effective in a subcritical regime.
It is, however,  not an appropriate tool in our energy-critical setting.

In \cite{Poc}, the second author considered 
the energy-critical defocusing NLW on $\R^d$, $d = 4, 5$, 
and proved almost sure global well-posedness below the energy space.
The main novel approach in \cite{Poc} is the {\it probabilistic perturbation theory}. 
See also  B\'enyi-Oh-Pocovnicu \cite{BOP2}.
One of the  key ingredients in applying
probabilistic perturbation theory
was a {\it probabilistic (a priori) energy bound}.
Here, the probabilistic  energy bound states that 
given any $T, \eps > 0$, there exists $\O_{T, \eps}\subset \O$
with $P(\O_{T, \eps}^c) < \eps$
such that, for all $\o \in \O_{T, \eps}$,
the solution $v^\o$ to the perturbed NLW  \eqref{v} 
(with the appropriate energy-critical powers for $d = 4, 5$)
satisfies
\begin{equation}
 \|(v^\omega(t), \dt v^\o(t))\|_{L^\infty([0,T],\dot{\mathcal{H}}^1(\R^d))}\leq C(T,\eps)
\label{PocEnergy}
 \end{equation}

\noi
for some $C(T, \eps) > 0$.
Such a probabilistic energy bound
was first established by Burq-Tzvetkov \cite{BT3} in the context of the (energy-subcritical) 
cubic NLW on $\T^3$.
In \cite{BT3, Poc}, 
\eqref{PocEnergy} was obtained
by estimating the growth of the (non-conserved) energy $E(v^\o)$ of 
the solution 
$v^\o$ to \eqref{v}
via probabilistic Strichartz estimates, Sobolev's inequality, 
and Gronwall's inequality.
Such an argument as in \cite{BT3, Poc}, however, does not hold
for the energy-critical defocusing quintic NLW \eqref{NLW} on $\R^3$.
In particular, 
the degree of the quintic nonlinearity is too high to close the argument.
See Remark 5.1 in \cite{Poc}.

Theorem \ref{THM:GWP} 
covers the missing case from \cite{LM, Poc}:
$p = 5$ and $d = 3$.
This corresponds to the 
energy-critical NLW in three
spatial dimensions,
and thus it is  of importance  from a physical point of view
as well as an analytical point of view. 
As in \cite{Poc}, 
the main approach 
to prove Theorem \ref{THM:GWP} is
the  probabilistic perturbation theory.
In the deterministic setting, perturbation theory has played an important role
in the study of the energy-critical NLS and NLW  \cite{CKSTT, KM}. 
It has also been effective in establishing
global well-posedness of NLS with a combined power-type nonlinearity
\cite{TVZ, KOPV}.
In our probabilistic approach, we view \eqref{v} as the defocusing quintic NLW
with a (random) perturbation given by 
$(v^\omega+z^\omega)^5-(v^\omega)^5$. 
Then,  smallness of the  perturbation 
comes from 
the probabilistic Strichartz estimates (Lemma \ref{LEM:Str})
satisfied by the random linear part $z^\o$.
In particular, 
by restricting the analysis to short time intervals, 
we can make the  perturbation small.

In applying perturbation theory in the probabilistic setting in \cite{Poc}, 
it was essential to have 
the  probabilistic energy bound \eqref{PocEnergy}.
As we pointed out above, however, 
the approach in \cite{BT3, Poc}
does not yield
a probabilistic energy bound 
\eqref{PocEnergy}
for the perturbed NLW \eqref{v} on $\R^3$.
Indeed, this is the main source of difficulty in establishing
Theorem \ref{THM:GWP}.
In order to resolve this issue,
we develop a more intricate analysis
that will allow us to obtain a
suitable replacement of the probabilistic energy bound \eqref{PocEnergy}.
More precisely,
we consider a sequence $\{v_N^\o\}_{N \geq 1, \text{ dyadic}}$ of smooth random approximating solutions
and establish 
 a {\it uniform} (in $N$) probabilistic energy bound
for 
$v_N^\o$.
See Proposition \ref{PROP:Penergy} below.
The main ingredient in the proof of 
Proposition \ref{PROP:Penergy}
is a new probabilistic estimate 
(Proposition \ref{PROP:infty}),
where we control the $L^\infty_t$-norm of random linear solutions.
We point out that 
we only prove 
a probabilistic energy bound, uniformly in $N$, 
for the approximating random solutions $v_N^\o$.
In particular, 
we do not know how to directly  prove 
a probabilistic energy bound  
\eqref{PocEnergy}
for the solution $v^\o$ to \eqref{v}.
Such a 
probabilistic energy bound for $v^\o$ follows 
as a corollary to the proof of Theorem \ref{THM:GWP}.
See Proposition \ref{PROP:aas} below.

Finally, 
the  uniform probabilistic energy bound (Proposition \ref{PROP:Penergy})
combined with the perturbation theory adapted to our setting (Proposition \ref{PROP:pLWP2})
yields Theorem \ref{THM:GWP}.

We conclude this introduction by stating several remarks.

\begin{remark} \label{REM:GWP}\rm
(i) The  uniqueness statement  in Theorem \ref{THM:GWP} holds
in the following sense.
The set $\O_{(u_0, u_1)} $ in Theorem \ref{THM:GWP}
can be written as $\O_{(u_0, u_1)} =\bigcup_{\eps>0}\O_\eps$
with $P(\O_\eps^c)<\eps$.
Given $\eps>0$, for all $\omega\in\O_\eps$ and any finite $T>0$, 
there exists a sequence of disjoint intervals $\{I_j\}_{j\in\mathbb N}$
covering $[-T,T]$ 
such that the solution $u^\omega$
is unique in some ball centered at $S(\cdot)(u_0^\omega, u_1^\omega)$
in $C(I_j,\dot H^1(\R^3))\cap L^5(I_j, L^{10}(\R^3))$
for all $j\in\mathbb N$.
The uniqueness part of Theorem \ref{THM:GWP}
is essentially contained in the local-in-time Cauchy theory
and we omit its proof.  See Theorem 5.3 in \cite{Poc}.

\smallskip

\noi
(ii)  
As in \cite{BT3, OQ, Poc}, we can enhance the statement in Theorem \ref{THM:GWP}
in the following sense.
Let  ${\bf u_0}:\O\to
\mathcal{H}^s(\R^3)$
be a map 
given by 
${\bf u_0}(\omega) : =  (u_0^\omega,u_1^\omega)$,
where $(u_0^\omega,u_1^\omega)$ 
is as  in \eqref{R1}.
Then, the map  ${\bf u_0}$
induces a probability measure $\mu = \mu_{(u_0, u_1)} = P\circ {\bf u_0}^{-1}$
on $\mathcal{H}^s(\R^3)$.
Arguing as in \cite{Poc}, 
we can show that there exists a set of $\mu$-full measure $\Sigma \subset \mathcal{H}^s(\R^3)$
such that 
(a)  for any $(\phi_0,\phi_1)\in \Sigma$,
there exists a unique global solution $u$ to  \eqref{NLW}  with initial data 
$\left(u,\pa_t u\right)\big|_{t=0}=(\phi_0, \phi_1)$
and (b) 
 $\mu\big(\Phi(t)(\Sigma)\big)  = 1 $ for any $t \in \R$,
 where
 $\Phi(t)$ denotes the solution map of \eqref{NLW}.
Namely, the measure of our initial data set $\Si$ does not become smaller
under the dynamics of \eqref{NLW}.

\smallskip

\noi
(iii)  As a byproduct of
the proof of Theorem \ref{THM:GWP}, we obtain
the probabilistic energy bound \eqref{PocEnergy}
for the solution $v^\o$ to \eqref{v}.
Then, 
by replacing the smooth cutoff function $\psi$
with the sharp cutoff function $\chi_{Q}$, 
we can also obtain  the probabilistic continuous dependence
of the solution map, and thus probabilistic Hadamard global well-posedness
in the sense of \cite{BT3, Poc}.
See Remark 1.4 in \cite{Poc}.

\smallskip

\noi
(iv)
The almost sure global well-posedness result
of the energy-critical wave equation on $\R^d$, $d = 4, 5$, in \cite{Poc}
holds for $s > 0$ when $d = 4$ and $s \geq 0$ when $d = 5$.
This is (almost) optimal in view of Remark \ref{REM:LWP} (i).
Theorem \ref{THM:GWP} on $\R^3$, however, 
holds only for $s > \frac{1}{2}$.
This regularity loss appears in establishing 
 a uniform probabilistic energy bound
for approximating random solutions
(Proposition \ref{PROP:Penergy}).
At this point, we do not know how to close this regularity gap.

\smallskip

\noi
(v)
In view of Theorem \ref{THM:GWP}, 
it is natural to consider 
the  problem of scattering
for \eqref{NLW}
in the probabilistic setting.
A key ingredient would be 
 to establish a probabilistic bound on the global space-time
Strichartz $L^5_tL^{10}_x$-norm
of the solution $v^\o$ to \eqref{v}.
The probabilistic 
 perturbation theory
used for 
Theorem \ref{THM:GWP}, 
however, only yields
a bound on the 
$L^5_tL^{10}_x$-norm
of the solution $v^\o$ on short time intervals
and does not allow us to establish a global space-time bound.
Thus, a new idea is needed
to prove probabilistic scattering (for large data).

As in the deterministic setting, 
there is no such difficulty in the small data case.
Indeed, 
L\"uhrmann-Mendelson \cite{LM}
proved a probabilistic small data scattering result for 
\eqref{NLW}
with large probability.
Moreover, even in the large data case, 
by considering the Wiener randomization on dilated cubes
as in \cite{BOP2}, 
one can establish a probabilistic scattering result for \eqref{NLW}
with large probability.
See \cite{BOP2} for details of such an argument.
It is worthwhile to note that 
these results hold  only with large probability, 
i.e.~not almost surely.

\end{remark}

This paper is organized as follows.
In Section \ref{SEC:2}, 
we introduce basic notations
and recall the deterministic Strichartz estimates.
Section \ref{SEC:3}
covers 
the  necessary probabilistic estimates.
In particular, 
Proposition \ref{PROP:infty}
is novel and plays an essential role in 
proving 
a uniform probabilistic energy bound
for  approximating random solutions 
(Proposition \ref{PROP:Penergy}) in Section \ref{SEC:4}.
In Section \ref{SEC:perturb}, 
we handle the deterministic component  of the proof of Theorem \ref{THM:GWP}.
Then, we present the proof of Theorem \ref{THM:GWP} in Section \ref{SEC:GWP}.

In view of  the time reversibility of the equation, 
we only consider positive times in the following.

\section{Notations}
\label{SEC:2}

We say that $u$ is a solution to the following nonhomogeneous wave equation:
\begin{equation}
\begin{cases}
\pa_{t}^2u-\Delta u+F=0\\
(u, \dt u)|_{t = t_0} = (\phi_0, \phi_1)
\end{cases}
\label{Wave}
\end{equation}

\noi
on a time interval $I$ containing $t_0$, 
if $u$ satisfies  the following Duhamel formulation:
\begin{equation}
u(t)=S(t-t_0)(\phi_0, \phi_1)
-\int_{t_0}^t \frac{\sin ((t-t')|\nabla|)}{|\nabla|}F(t')dt'
\label{Duhamel}
\end{equation}

\noi
for $t \in I$.
Here, $S(\cdot)$ denotes the linear propagator defined in \eqref{Zlinear}.
We now recall the Strichartz estimates for wave equations on $\R^3$.
 We say that $(q,r)$ is a $s$-wave admissible pair
if $q\geq 2$, $2\leq r<\infty$,
\[\frac{1}{q}+\frac{1}{r}\leq \frac{1}{2},
\quad \text{and}\quad 
\frac 1q+\frac 3r=\frac 32-s.\]

\noi
Then, we have the following Strichartz estimates.
See \cite{Ginibre, Lindblad, Keel}
for more discussions on the  Strichartz estimates. 

\begin{lemma}
\label{LEM:Strichartz}
Let $s>0$. 
Let  $(q,r)$ and $(\tilde{q},\tilde{r})$ be $s$- and $(1-s)$-wave admissible pairs,
respectively. 
Then, we have
\begin{align}\label{Strichartz}
\|(u, \dt u)\|_{L^\infty_t(I;\dot{\mathcal{H}}^s_x)}+\|u\|_{L^q_t(I; L^r_x)}\lesssim
\|(\phi_0, \phi_1) \|_{\dot{\mathcal{H}}^{s}}+\|F\|_{L^{\tilde{q}'}_t(I; L^{\tilde{r}'}_x)}
\end{align}

\noi
for all solutions  $u$ to \eqref{Wave} on a time interval $I \ni t_0$.

\end{lemma}

\noi
In our argument, 
we will only use 
the following
wave admissible pairs: $\big(5, 10)$ with $s = 1$
and 
 $(\infty,2)$ with $s = 0$.
For simplicity,  we often denote the space $L^q_t(I; L^r_x)$  by $L^q_IL^r_x$
or  $L^q_TL^r_x$ if $I=[0,T]$.

Next, we briefly go over 
the Littlewood-Paley theory.
Let $\varphi:\R \to [0, 1]$ be a smooth  bump function supported on $[-\frac{8}{5}, \frac{8}{5}]$ 
and $\varphi\equiv 1$ on $[-\frac 54, \frac 54]$.
Given dyadic $N \geq1$, 
we set $\varphi_1(\xi) = \varphi(|\xi|)$
and 
\[\varphi_N(\xi) = \varphi\big(\tfrac{|\xi|}N\big)-\varphi\big(\tfrac{2|\xi|}N\big)\]

\noi
for $N \geq 2$.
Then, we define the Littlewood-Paley projection $\P_N$
as the Fourier multiplier operator with symbol $ \varphi_N$.
Moreover, we define $\P_{\leq N}$ and $\P_{\geq N}$
by 
 $\P_{\leq N} = \sum_{1\leq M \leq N} \P_M$ and $\P_{\geq N} = \sum_{M\geq N} \P_M$.

Lastly, recall  Bernstein's inequality:
\begin{equation}
\label{Bern1}
\|\pP_{\leq N} f\|_{L^q(\R^3)} \lesssim N^{\frac{3}{p}-\frac{3}{q}}\|\pP_{\leq N} f\|_{L^p(\R^3)}, \quad
1\leq p \leq q \leq \infty.
\end{equation}

\noindent
As an immediate corollary of \eqref{Bern1}, we have, 
for all $n\in\Z^3$,
\begin{equation}
\label{Bern2}
\|\psi(D -n) \phi\|_{L^q(\R^3)} \lesssim \|\psi(D-n)  \phi \|_{L^p(\R^3)}, \qquad
1\leq p \leq q \leq \infty.
\end{equation}

\section{Probabilistic estimates}
\label{SEC:3}

In this section,  we first review some basic properties of randomized functions.
Then, we present the main new probabilistic estimate
(Proposition \ref{PROP:infty}), controlling the $L^\infty_t$-norm
of random linear solutions.

First recall the following probabilistic estimate.
See \cite{BTI} for the proof.

\begin{lemma}\label{LEM:HC}
Let $\{g_n\}_{n\in \Z^3}$ be a sequence of mean zero complex-valued,
 random variables  
such that $g_{-n}=\overline{g_n}$ for all 
$n\in\Z^3$.
With $\mathcal{I}$  as in \eqref{Index}, 
assume that $g_0$, $\Re g_n$, and $\Im  g_n$,
$n\in\mathcal I$, 
are independent. 
Moreover, 
assume that 
\eqref{cond} is satisfied.
Then, there exists $C>0$ such that the following holds:
\begin{equation*}
\Big\|\sum_{n\in\Z^3}g_n(\omega)c_n\Big\|_{L^p(\Omega)}\leq C\sqrt{p} \|c_n\|_{\ell^2_n(\Z^3)}
\end{equation*}

\noi
for any $p\geq 2$
and any sequence $\{c_n\} \in \ell^2 (\mathbb{Z}^3)$ satisfying $c_{-n}=\overline{c_n}$
for all $n\in\Z^3$.
\end{lemma}

Next, we recall the local-in-time probabilistic Strichartz estimates.

\begin{lemma}[Proposition 2.3 in \cite{Poc}]\label{LEM:Str}
Given a pair   $(u_0, u_1)$ of real-valued functions defined on $\R^3$,
let $(u_0^{\omega}, u_1^\omega)$ be the Wiener randomization defined in \eqref{R1}, satisfying \eqref{cond}. Let $I=[a,b]\subset \R$ be a compact time interval.

\smallskip 
\noindent
{\rm(i)} If $(u_0, u_1) \in \dot{\mathcal{H}}^0(\R^3)$,
then given $1\leq q<\infty$ and $2\leq r<\infty$, there exist $C,c>0$ such that
\begin{align*}
P\left(\left\|S(t)(u_0^\omega,u_1^\omega)\right\|_{L^q_t(I; L^r_x)}>\ld\right)
\leq C\exp\Bigg(-c\frac{\ld^2}{|I|^{\frac2q} \|(u_0, u_1)\|_{\dot{\mathcal{H}}^0}^2}\Bigg).
\end{align*}

\noindent
{\rm (ii)} 
If $(u_0, u_1) \in \mathcal{H}^s(\R^3)$,
 then given $1\leq q< \infty$, $2\leq r\leq \infty$,
there exist $C,c>0$ such that
\begin{align*}
P\left(\left\|S(t)(u_0^\omega,u_1^\omega)\right\|_{L^q_t(I; L^r_x)}>\ld\right)
\leq C\exp\Bigg(-c\frac{\ld^2}{\max\left(1, |a|^2, |b|^2\right)|I|^{\frac2q}
\|(u_0, u_1)\|_{\mathcal{H}^s}^2}\Bigg)
\end{align*}

\noi
for  \textup{(ii.a)} $s = 0$ if $r < \infty$
and \textup{(ii.b)}  $s > 0$ if $r = \infty$.

\end{lemma}

Lemma \ref{LEM:Str} plays an essential role
in the proof of Theorem \ref{THM:LWP}.
The proof of Lemma \ref{LEM:Str}
follows
from  Lemma \ref{LEM:HC} and \eqref{Bern2}.
See \cite{LM, Poc} for details.


The following proposition allows us to obtain a probabilistic estimate involving the 
$L^\infty_t$-norm
and plays an important role in establishing a
probabilistic energy bound.  See Proposition \ref{PROP:Penergy} below.

Define $\wt S(t)$ by 
\begin{equation}
 \wt S(t)(u_0, u_{1}) 
:=  -\frac{|\nb|}{\jb{\nb}}\sin (t|\nb|) u_{0} + \frac{ \cos (t|\nb|)}{\jb{\nb}}u_{1}.
\label{wlinear}
\end{equation}

\noi
Namely, we have $\dt   S(t)(u_0, u_{1}) =  \jb{\nb}\wt S(t)(u_0, u_{1}) $.

\begin{proposition}\label{PROP:infty}
Given a pair   $(u_0, u_1)$ of real-valued functions defined on $\R^3$,
let $(u_0^{\omega}, u_1^\omega)$ be the Wiener randomization defined in \eqref{R1}, satisfying \eqref{cond}. 
Let $T> 0$ and $S^*(t) = S(t)$ or $\wt S(t)$ defined in \eqref{Zlinear} and \eqref{wlinear}, 
respectively.
Then, for  $2\leq r\leq\infty$, we have 
\begin{align}
P\Big( \|S^*(t) (u_0^\o, u_1^\o) & \|_{L^\infty_t([0, T];  L^r_x(\R^3))} > \ld\Big)\notag \\
&  \leq C 
(1+T) \exp \Bigg( -c \frac{\ld^2}{\max(1, T^2 )\|(u_0, u_1)\|^2_{\mathcal{H}^\eps (\R^3)}}\Bigg)
\label{infty00}
\end{align}

\noi
for any $\eps > 0$, 
where the constants $C$ and $c$ depend only on $r$ and $\eps$.

\end{proposition}

Proposition \ref{PROP:infty}  follows as a corollary to the following lemma.
Let  $S_+(t)$ and $S_-(t)$ be the linear propagators for
the half wave equations defined by 
\[ S_{\pm}(t) \phi := 
\mathcal{F}^{-1} \big(e^{\pm i|\xi|t} \ft \phi (\xi)\big).\]

\noi
Given $\phi \in H^s(\R^3)$, 
we define its randomization  $\phi^\o$ by 
\begin{align*}
\phi^\omega := 
\sum_{n \in \Z^d} g_{n, 0} (\omega) \psi(D-n) \phi
\end{align*}

\noi
as in  the first component of \eqref{R1}.
Then, we have the following tail estimate on 
the size of 
$S_\pm(t) \phi^\o$
over a time interval of length 1.

\begin{lemma}\label{LEM:infty}
Let $j \in \mathbb{N}\cup\{0\}$
and   $2\leq r \leq \infty$.
Given 
 any $\eps > 0$, 
there exist constants $C, c>0$, 
depending only on $r$ and $\eps$, 
such that
\begin{align}
 P\Big( \|S_\pm(t) \phi^\o\|_{L^\infty_t([j, j+1];  L^r_x(\R^3))} > \ld\Big)
&  \leq C \exp \Bigg( -c \frac{\ld^2}{\|\phi\|^2_{H^\eps (\R^3)}}\Bigg), 
 \label{infty0a}\\
 P\Bigg( \bigg\|\frac{\sin(t |\nb|)}{|\nb|} \phi^\o \bigg\|_{L^\infty_t([j, j+1];  L^r_x(\R^3))} > \ld\Bigg)
&  \leq C \exp \Bigg( -c \frac{\ld^2}{\max(1, j^2)\|\phi\|^2_{H^{\eps-1} (\R^3)}}\Bigg).
 \label{infty0b}
\end{align}

\end{lemma}

Assuming Lemma \ref{LEM:infty}, 
we first present the proof of Proposition \ref{PROP:infty}.

\begin{proof}[Proof of Proposition \ref{PROP:infty}]

We only consider the case $S^*(t) = S(t)$ and $T\geq 1$.
When $S^*(t) = \wt S(t)$, \eqref{infty00} holds without the factor $T^2$ in the exponent.
By subadditivity and Lemma \ref{LEM:infty}, we have
\begin{align*}
P\Big( \|S(t)  (u_0^\o, u_1^\o)& \|_{L^\infty_t([0, T];  L^r_x(\R^3))} > \ld\Big)
 \leq 
P\Big(\max_{j = 0, \dots, [T]}\|S(t) (u_0^\o, u_1^\o)\|_{L_t^\infty([j, j+1]; L^r_x(\R^3))} > \ld\Big)\\
& \leq 
\sum_{j = 0}^{[T]}
P\Big(\|S(t) (u_0^\o, u_1^\o)\|_{L_t^\infty([j, j+1]; L^r_x(\R^3))} > \ld\Big)\\
& \leq 
\sum_{j = 0}^{[T]}
P\bigg(\|\cos (t|\nb|) u_0^\o\|_{L_t^\infty([j, j+1]; L^r_x(\R^3))} > \frac{\ld}{2}\bigg)\\
& \hphantom{XXX}
+
\sum_{j = 0}^{[T]}
P\Bigg(\bigg\|\frac{\sin (t|\nb|)}{|\nb|} u_1^\o\bigg\|_{L^\infty([j, j+1]; L^r_x(\R^3))} > \frac{\ld}{2}\Bigg)\\
& \leq C([T]+1) \exp \Bigg( -c \frac{\ld^2}{T^2 \|(u_0, u_1) \|^2_{\mathcal{H}^\eps(\R^3)}}\Bigg).
\end{align*}

\noi
Here, $[T]$ denotes the integer part of $T$.
\end{proof}

Finally, we prove Lemma \ref{LEM:infty}.

\begin{proof}[Proof of Lemma \ref{LEM:infty}]
 We first prove \eqref{infty0a}.
Define $z_\pm^\o(t)$ and $z_\pm(t)$ by 
\[z_\pm^\o(t) := S_\pm(t) \phi^\o\quad \text{and}
\quad
z_\pm(t) := S_\pm(t) \phi.\]

\noi
{\bf Part 1\,(a):}
We first consider the case $r < \infty$.
Without loss of generality, assume $j = 0$. 
For $k \in \mathbb{N} \cup \{0\}$, 
let $\{ t_{\l, k}:\l = 0, 1,\dots, 2^k\} $ be $2^{k}+1$ equally spaced points on $[0, 1]$, 
i.e.~$t_{0,k}=0$ and $t_{\l, k} - t_{\l-1, k} = 2^{-k}$ for $\l =1, \dots, 2^k$.
Then, given $t \in [0, 1]$, 
we have
\begin{align}
  z_\pm^\o(t)
= \sum_{k = 1}^\infty \big(z_\pm^\o(t_{\l_k, k}) - z_\pm^\o(t_{\l_{k-1}, k-1})\big)
+ z_\pm^\o(0)
\label{infty2}
\end{align}

\noi
for some $\l_k = \l_k(t) \in \{0, \dots, 2^k\}$.\footnote{Think of the binary expansion of this given $t \in [0, 1]$.
Then, $t_{\l_k, k}$ can be given by the partial sum of this binary expansion up to order $k$.}

We consider the $L^r_x$-norm with the (inhomogeneous) Littlewood-Paley decomposition.
Then, by the square function estimate and Minkowski's integral inequality, we have
\begin{align}
 \| z_\pm^\o(t)\|_{L^r_x}
 \sim \bigg\|\Big(\sum_{\substack{N \geq 1\\\text{dyadic}}}
 |\P_N z_\pm^\o(t) |^2 \Big)^\frac{1}{2}\bigg\|_{L^r_x}
\leq
\Big(\sum_{\substack{N \geq 1\\\text{dyadic}}}
 \|\P_N z_\pm^\o(t) \|_{L^r_x}^2 \Big)^\frac{1}{2}.
\label{infty3}
 \end{align}

\noi
Then, from \eqref{infty2} and \eqref{infty3}, we have
\begin{align*}
 \| z_\pm^\o\|_{L^\infty_t([0, 1]; L^r_x)}
\les
\bigg(\sum_{\substack{N \geq 1\\\text{dyadic}}}
\Big( \sum_{k =1}^\infty
 \max_{0\leq \l_k \leq 2^k} 
\big\|\P_N 
 \big(z_\pm^\o(t_{\l_k, k}) - z_\pm^\o(t_{\l'_{k-1}, k-1})\big)
\big\|_{L^r_x} \Big)^2\bigg)^\frac{1}{2}
+ \|z_\pm^\o(0)\|_{L^r_x},
 \end{align*}

\noi
where $t_{\l'_{k-1}, k-1}$ is one of the $2^{(k-1)}$+1 equally spaced points such that 
\begin{align}
|t_{\l_k, k} - t_{\l'_{k-1}, k-1}| \leq 2^{-k}.
\label{infty3a}
\end{align}

\noi
Hence, for  $p \geq 2$, we have
\begin{align}
\Big(\E\big[  \| z_\pm^\o& \|_{L^\infty_t([0, 1]; L^r_x)}\big]^p\Big)^\frac{1}{p} \notag \\
& \les
\Bigg(\sum_{\substack{N \geq 1\\\text{dyadic}}}
\bigg( \sum_{k = 1}^\infty
\Big(\E \Big[
 \max_{0\leq \l_k \leq 2^k} 
 \big\|\P_N 
 \big(z_\pm^\o(t_{\l_k, k}) - z_\pm^\o(t_{\l'_{k-1}, k-1})\big)
 \big\|_{L^r_x}
 \Big]^p
  \Big)^\frac{1}{p}\bigg)^2\Bigg)^\frac{1}{2}\notag \\
& + 
\Big(\E\big[ \|z_\pm^\o(0)\|_{L^r_x}\big]^p\Big)^\frac{1}{p}.
\label{infty4}
\end{align}

\noi
Note that it follows from Lemma \ref{LEM:HC}
and \eqref{Bern2}
that the second term on the right-hand side of \eqref{infty4} can be bounded by
\begin{align}
\Big(\E\big[ \|z_\pm^\o(0)\|_{L^r_x}\big]^p\Big)^\frac{1}{p}
\les \sqrt p \|\phi\|_{L^2_x}
\label{infty5}
\end{align}

\noi
for $p \geq r$. 
	
In the following, we first estimate	
\begin{align*}
I_N: =  \sum_{k = 1}^\infty
\bigg(\E \Big[
 \max_{0\leq \l_k \leq 2^k} 
 \big\|\P_N 
 \big(z_\pm^\o(t_{\l_k, k}) - z_\pm^\o(t_{\l'_{k-1}, k-1})\big)
 \big\|_{L^r_x}
 \Big]^p
  \bigg)^\frac{1}{p}
\end{align*}
	
\noi
for each  dyadic $N \geq 1$.
Let 
\begin{align} 
q_k  := \max (\log 2^k, p, r)
\sim \log 2^k + p.
\label{infty5a}
\end{align}

\noi
Then, we have
\begin{align}
I_N 
& \leq  \sum_{k = 1}^\infty
\bigg(
 \sum_{\l_k = 0}^{2^k} 
 \E 
 \big\|\P_N 
 \big(z_\pm^\o(t_{\l_k, k}) - z_\pm^\o(t_{\l'_{k-1}, k-1})\big)
 \big\|_{L^r_x}^{q_k}
  \bigg)^\frac{1}{q_k}\notag\\
\intertext{Noting  that $(2^k+1)^\frac{1}{q_k} \les 1$
and applying  Lemma \ref{LEM:HC}, 
}
& \les  \sum_{k = 1}^\infty
 \max_{0\leq \l_k \leq 2^k} 
\Big(\E
 \big\|\P_N 
 \big(z_\pm^\o(t_{\l_k, k}) - z_\pm^\o(t_{\l'_{k-1}, k-1})\big)
 \big\|_{L^r_x}^{q_k}
  \Big)^\frac{1}{q_k} \notag\\
& \les  \sum_{k = 1}^\infty
\sqrt{q_k}
 \max_{0\leq \l_k \leq 2^k} 
\Big\|
 \big\| \P_N \psi(D-n) 
 \big(z_\pm(t_{\l_k, k}) - z_\pm(t_{\l'_{k-1}, k-1})\big)
 \big\|_{L^r_x}
 \Big\|_{\l^2_{|n|\sim N}}\notag
 \intertext{By \eqref{Bern2}, we have}
& \les  \sum_{k = 1}^\infty
\sqrt{q_k}
 \max_{0\leq \l_k \leq 2^k} 
\Big\|
 \big\|\P_N \psi(D-n) 
 \big(z_\pm(t_{\l_k, k}) - z_\pm(t_{\l'_{k-1}, k-1})\big)
 \big\|_{L^2_x}
 \Big\|_{\l^2_{|n|\sim N}}.
\label{infty6}
\end{align}

Then, by \eqref{infty3a} we have 
\begin{align}
 \big\|\P_N  \psi(D-n) 
&   \big(z_\pm(t_{\l_k, k}) - z_\pm(t_{\l'_{k-1}, k-1})\big)
 \big\|_{L^2_x} \notag \\
&   = \bigg(\int_{|\xi|\sim N}
 \Big|e^{\pm i|\xi| t_{\l_k, k} }
-  e^{\pm i|\xi|  t_{\l'_{k-1}, k-1}}
 \Big|^2 |\psi(\xi - n) \ft \phi(\xi)|^2 d\xi \bigg)^\frac{1}{2}\notag\\
&  \les
\min(1, 2^{-k }N)
 \big\|\P_N  \psi(D-n) 
\phi \big\|_{L^2_x}.
\label{infty7}
\end{align}
	
\noi
Hence, from \eqref{infty6} and \eqref{infty7}, it follows that
\begin{align}
I_N 
& \les  \sum_{k = 1}^\infty
\sqrt{q_k}\min(1, 2^{-k }N)
\|\P_N \phi\|_{L^2_x}. \label{infty8}
\end{align}

Now, we separate the summation into $2^{-k}N \geq 1$
and $2^{-k}N <  1$
and estimate the contribution from each case.
Note that
 from \eqref{infty5a}, we have
\begin{align}
\sqrt q_k \les \sqrt{\log 2^k}\cdot \sqrt p.
\label{infty5x}
\end{align}

\medskip

\noi
$\bullet$ {\bf Case 1:} $2^{-k}N \geq  1$.

In this case,
 from \eqref{infty5x}, we have
\begin{align*}
\sqrt q_k \les \sqrt{\log N}\cdot \sqrt p
\end{align*}

 \noi
 Hence, we have
 \begin{align}
 \eqref{infty8}
& \leq C_r (\log N)^\frac{3}{2}
\sqrt{p}
\|\P_N \phi\|_{L^2}
 \leq C_{r, \eps}  
\sqrt{p}
\|\P_N \phi\|_{H^\eps} \label{infty9a}
\end{align}

\noi
for any $\eps > 0$.

\medskip

\noi
$\bullet$ {\bf Case 2:} $2^{-k}N <  1$.

From \eqref{infty5x}, we have 
\begin{align}
 \eqref{infty8}
& \leq C_r \sqrt{p}  \sum_{k \ges  \log N}^\infty
(\log 2^k)^\frac{1}{2}
2^{-k }N
\|\P_N \phi\|_{L^2}
 \les  C_r \sqrt{p} 
( \log N)^\frac{1}{2}
\|\P_N \phi\|_{L^2}\notag\\
& \leq 
C_{r, \eps} \sqrt{p}
\|\P_N \phi\|_{H^\eps} \label{infty9b}
\end{align}

\noi
for any $\eps > 0$.

Finally, putting \eqref{infty4},  \eqref{infty5}, \eqref{infty8}, \eqref{infty9a}, and \eqref{infty9b}
together, we obtain 
\begin{align*}
\Big(\E\big[  \| z^\o \|_{L^\infty_t([0, 1]; L^r_x)}\big]^p\Big)^\frac{1}{p}
 \leq C_{r, \eps}
\sqrt{p}
\|\phi\|_{H^\eps_x} 
\end{align*}

\noi
for all $p \geq r$ and $\eps > 0$.
The rest follows from a standard argument
using Chebyshev's inequality.

\smallskip

\noi
{\bf Part 1\,(b):}
Next, we consider the case $r = \infty$.
Then, it follows from Sobolev embedding that, 
given any $\eps>0$,  there exists  large $\tilde r \gg1$ with $\eps \wt{r} > 3$
such that 
\begin{align*}
P\Big( \|S_\pm(t)  \phi^\o\|_{L^\infty_t([j, j+1];  L^\infty_x(\R^3))} > \ld\Big)
 \leq P\Big( \|\jb{\nb}^\eps S_\pm(t) \phi^\o\|_{L^\infty_t([j, j+1];  L^{\wt{r}}_x(\R^3))} > C \ld\Big).
\end{align*}

\noi
Then, the rest follows from the argument in Part 1 (a).

\smallskip

\noi
{\bf Part 2:} Next, we briefly discuss how to prove \eqref{infty0b} when $r < \infty$.
Letting 
\[Z^\o(t): = \frac{\sin(t |\nb|)}{|\nb|} \phi^\o 
\quad \text{and}
\quad 
Z(t): = \frac{\sin(t |\nb|)}{|\nb|} \phi \]
and repeating the argument in Part 1
(but on $[j, j + 1]$ instead of $[0, 1]$), we have 
\begin{align*}
\Big(\E\big[  \| Z^\o& \|_{L^\infty_t([j, j+1]; L^r_x)}\big]^p\Big)^\frac{1}{p} \notag \\
& \les
\Bigg(\sum_{\substack{N \geq 1\\\text{dyadic}}}
\bigg( \sum_{k = 1}^\infty
\Big(\E \Big[
 \max_{0\leq \l_k \leq 2^k} 
 \big\|\P_N 
 \big(Z^\o(t_{\l_k, k}) -Z^\o(t_{\l'_{k-1}, k-1})\big)
 \big\|_{L^r_x}
 \Big]^p
  \Big)^\frac{1}{p}\bigg)^2\Bigg)^\frac{1}{2}\notag \\
& + 
\Big(\E\big[ \|Z^\o(j)\|_{L^r_x}\big]^p\Big)^\frac{1}{p}
= : \I + \II.
\end{align*}

\noi
When $j = 0$, then we have $\II = 0$.
When $j \geq 1$, we argue as in the proof of Proposition 2.3 (ii) in \cite{Poc}
and obtain 
\begin{align}
\II \les \sqrt p \max (1, j) \|\phi\|_{H^{-1}}
\label{infty10a}
\end{align}

\noi
for $p \geq r$. 
As for $\I$, we simply repeat the computations in Part 1
with a modification in \eqref{infty7}:
\begin{align*}
 \big\|\P_N  \psi(D-n) 
&   \big(Z(t_{\l_k, k}) - Z(t_{\l'_{k-1}, k-1})\big)
 \big\|_{L^2_x} \notag \\
&   \sim \Bigg(\int_{|\xi|\sim N}
 \bigg|\frac{e^{\pm i|\xi| t_{\l_k, k} }
-  e^{\pm i|\xi|  t_{\l'_{k-1}, k-1}}}{\xi}
 \bigg|^2 |\psi(\xi - n) \ft \phi(\xi)|^2 d\xi \Bigg)^\frac{1}{2}\notag\\
&  \les
\min(N^{-1}, 2^{-k})
 \big\|\P_N  \psi(D-n) 
\phi \big\|_{L^2_x}.
\end{align*}

\noi
This modification yields
\begin{align}
\I \les C_{r, \eps} \sqrt p  \|\phi\|_{H^{\eps-1}}.
\label{infty10b}
\end{align}

\noi
Then, the desired estimate \eqref{infty0b} follows
from \eqref{infty10a} and \eqref{infty10b}.
\end{proof}


\section{Uniform probabilistic  energy bound
for approximating solutions}
\label{SEC:4}

Let $(u_0, u_1)\in \mathcal{H}^s(\R^3)$ with $\frac 12 < s< 1$.
Given $N \geq 1$ dyadic, 
define $u_{j, N}^\o$, $j = 0, 1$,  by 
\begin{equation}
   u_{j, N}^\o := \P_{\leq N} u_{j}^\o = \sum_{n \in \Z^3}g_{n,j}(\omega)\P_{\leq N} \psi(D-n) u_j.
\label{cutoffdata}
\end{equation}

\noi
Note that  we have
$(u_{0, N}^\o, u_{1, N}^\o) \in \mathcal{H}^\infty(\R^3)$.
Let $u_N$ be the smooth global-in-time solution to \eqref{NLW}
with initial data
\begin{align*}
(u_N, \dt u_N)|_{t = 0} = (u_{0, N}^\o, u_{1, N}^\o),
\end{align*}

\noi
and  denote by $z_N = z_N^\o$ and $v_N = v_N^\o$ 
the linear and nonlinear parts of $u_N$.
Namely, 
\begin{equation}
z_N (t) := S(t) (u_{0, N}^\o, u_{1, N}^\o)
\qquad \text{and}
\qquad 
v_N := u_N - z_N.
\label{Penergy0}
\end{equation}

\noi 
In particular, $v_N$ is the 
smooth global solution
to the following perturbed NLW:
\begin{align}
\begin{cases}
 \dt^2 v_N - \Dl v_N + ( v_N+z_N)^5 = 0,\\
 (v_N, \dt v_N) |_{t = 0} = (0, 0).
\end{cases}
\label{NLW1}
\end{align}

\noi
It follows from the conservation of the energy of $u_N$
and the unitarity of the linear propagator that 
we have $\| (v_N^\o,\partial_t v_N^\o) \|_{L^\infty(\R;  \dot{ \mathcal{H}}^1(\R^3))} \leq C(N, \o)  < \infty$
for each $N \in \mathbb{N}$.
There is, however, no uniform control on the size of $v_N$, independent of $N$, 
since the $\dot{H}^1$-norm of $z_N$ tends to infinity almost surely  as $N \to \infty$.

The following proposition 
establishes a probabilistic energy bound on $v_N$,
{\it independent} of dyadic $N \geq 1$,
and plays an important role in the proof of Theorem  \ref{THM:GWP}.

\begin{proposition}\label{PROP:Penergy}
Let $ s \in (\frac 12, 1)$ and $N \geq 1$ dyadic.
Given $T,  \eps > 0$, there exists $\wt{ \O}_{N, T, \eps}\subset \O$
such that 
\begin{itemize}
\item[(i)] $P(\wt \O_{N, T, \eps}^c) < \eps$, 
\item[(ii)]  
There exists a finite constant $C(T, \eps, \|(u_0, u_1)\|_{{\mathcal{H}}^s(\R^3)}) > 0$
such that
the following energy bound holds:
\begin{align}
\sup_{t \in [0, T]}\| (v^\o_N  (t),\partial_t v^\o_N  (t)) \|_{ \mathcal{H}^1(\R^3)} \leq C(T, \eps,  \|(u_0, u_1)\|_{{\mathcal{H}}^s(\R^3)}),
\label{Penergy1}
\end{align}

\noi
for all   solutions $v^\o_N$ to \eqref{NLW1} 
with  $\o \in \wt{\O}_{N, T, \eps}$.
\end{itemize}
	
\noi	
Note that  the constant $C(T, \eps,  \|(u_0, u_1)\|_{{\mathcal{H}}^s(\R^3)})$ 
is independent of dyadic $N \geq 1$.
\end{proposition}

\noi
\begin{proof}

First, note that it suffices to prove
\begin{align}
\sup_{t \in [0, T]}\| (v^\o_N  (t), \partial_t v^\o_N  (t)) \|_{ \dot{\mathcal{H}}^1(\R^3)} \leq C\big(T, \eps,  \|(u_0, u_1)\|_{{\mathcal{H}}^s(\R^3)}\big).
\label{Penergy1a}
\end{align}

\noi
Indeed, \eqref{Penergy1} follows from  \eqref{Penergy1a}
and 
\begin{align}
\|v_N^\omega (t)\|_{L^2_x(\R^3)}
&=\bigg\|\int_0^t\pa_t v_N^\omega(t')dt'\bigg\|_{L^2_x} 
\leq T\big\|\pa_t v_N^\omega\big\|_{L^\infty_T L^2_x}
\leq C\big(T,\eps,\|(u_0,u_1)\|_{\mathcal{H}^s(\R^3))}\big).
\label{L2}
\end{align}

Let $z_N(t)$ be as in \eqref{Penergy0}
and  $\wt z(t) = \wt z^\o_N(t) := \wt S(t)(u^\o_{0, N}, u^\o_{1, N}) $
with $\wt S(t)$ defined in \eqref{wlinear}. 
Let $\dl > 0$ sufficiently small such that $\frac 12 + \dl < s$.
For fixed $T ,  \eps > 0$, 
we define $\wt \O_{N, T, \eps}$
by 
\[ \wt \O_{N, T, \eps}
= \big\{\o:\,  \|z_N^\o \|_{L^{10}_{T, x}}^{10} + \|z_N^\o \|_{L^\infty_TL^6_x}^6 
+ \|z_N^\o\|_{L^\infty_{T, x}}^2  
+ \|\wt z^\o_N \|_{L^6_{T, x}}^6\\
+ \big\|\jb{\nb}^{s-\delta}  \wt z_N^\o\big\|_{L^\infty_{T, x}} \leq \ld\big\},\]

\noi
where 
$\ld = \ld\big(T, \eps, \|(u_0, u_1)\|_{\mathcal{H}^s(\R^3)}\big) > 0$ is chosen such that 
$P(\wt \O_{N, T, \eps}^c) < \eps$.
Note that the existence of such $\ld(T, \eps)$
is guaranteed by
Lemma \ref{LEM:Str} 
and Proposition \ref{PROP:infty}.
Moreover, $\ld(T, \eps)$ can be chosen to be independent of $N$.

In the following, we prove 
\begin{align}
\sup_{t \in [0, T]}E (v^\o_N  (t))  \leq C\big(T, \eps,  \|(u_0, u_1)\|_{\mathcal{H}^s(\R^3)}\big)
\label{Penergy2}
\end{align}

\noi
for $\o \in \wt \O_{N, T, \eps}$.
Then, \eqref{Penergy1a} follows from the coercivity of the energy $E$.

For simplicity,  we denote $v_N^\o$ and $z_N^\o$ by $v$ and $z$,
in the following.
By differentiating $E(v)$ in time, we have
\begin{align}
\frac{d}{dt} E(v)(t)
& = \int_{\R^3}\dt v  (\dt^2 v  - \Dl v + v^5) dx
=
- \int_{\R^3}\dt v  \big((z+v)^5  - v^5\big) dx\notag \\
& = 
- \int_{\R^3}\dt v (5 z v^4 +  \mathcal N(z, v) ) dx, \notag
\end{align}

\noi
where
\[\N(z, v): = 10 z^2 v^3 + 10 z^3 v^2 + 5 z^4 v + z^5.\]

\noi
By integrating in time, we have
\begin{align}
E(v)(t)
&  = \underbrace{E(v)(0)}_{=0}
- \int_0^t\int_{\R^3}
\dt v (t') \big[5 z(t') v(t')^4 +\mathcal N(z, v)(t') \big] dx dt' \notag\\
&  = - 
\int_{\R^3} \int_0^t z (t') \dt ( v(t') ^5) dt' dx
- \int_0^t\int_{\R^3}
\dt v(t') \mathcal N(z, v)(t')  dx dt' \notag \\
& =:\I(t) +\II(t), \label{E1}
\end{align}
	
\noi
for $t \in [0, T]$.
Noting that
\[|\N(z, v)(t')| \les |z(t')^2 v(t')^3| + |z(t')|^5,\]

\noi
we have 
\begin{align}
|\II(t)| 
& \les \int_0^t \|\dt v(t)\|_{L^2_x} \|z(t') \|^2_{L^\infty_x}\|v(t') \|_{L^6_x}^3
dt'  
+ \int_0^T \|\dt v(t')\|_{L^2_x} \|z(t') \|^5_{L^{10}_x}
dt' \notag \\ 
& \les \big(1+ \|z \|^2_{L^\infty_{T, x}} \big)\int_0^t E(v)(t') dt' 
+ \|z \|^{10}_{L^{10}_{T, x}}.
\label{E2}
\end{align}

Next, we control the term $\I(t)$.
Note that $v(0) \equiv 0$ and $v= v^\o_N$ is smooth, both in $x$ and $t$.
Then, by integration by parts in time, we have
\begin{align}
\I(t) = 
 - \int_{\R^3}  z (t)  v(t) ^5 dx
+ \int_{\R^3} \int_0^t \dt z (t')  v(t') ^5dt' dx
=:\I_1(t) +\I_2(t).
\label{E3}
\end{align}

\noi
As for the first term $\I_1(t)$, we bound it by
\begin{align}
|\I_1 (t)|
\les   a^{-6}\|z (t)\|_{L^6_x}^6  + a^\frac{6}{5} \|v(t)\|^6_{L^6_x}
\les a^{-6} \|z\|_{L^\infty_T L^6_x}^6 
+ a^\frac{6}{5} E(v)(t).
\label{E4}
\end{align}

\noi
for some small constant $a > 0$ (to be chosen later).

It remains to estimate the second term $\I_2(t)$ in \eqref{E3}.
Noting that $z(t)$ solves the linear wave equation, 
we have
\begin{align}
\I_2 (t) & =  \int_{\R^3} \int_0^t \dt z (t')  v(t') ^5dt' dx
=  \int_0^t 
\int_{\R^3} \jb{\nabla} \wt z (t') \cdot  v(t') ^5 dx dt'.
\label{E4a}
\end{align}

\noi
Define $\mathcal{I}(t)$ by 
\begin{align*}
\mathcal I (t) :=  \int_{\R^3} \jb{\nabla} \wt z (t)  \cdot v(t) ^5 dx 
& \sim 
\sum_{k = -1}^1\sum_{\substack{M \geq 1\\ \text{dyadic}}}
M \int_{\R^3} \P_{2^k M} \wt z (t)  \P_M\big[ v(t) ^5\big] dx, 
\end{align*}

\noi
with the understanding that  $\P_{2^{-1}} = 0$.

\noi
$\bullet$ {\bf Case 1:} $M = 1$.

By Young's and Bernstein's inequalities, we have 
\begin{align}
|\I_2(t)| 
\les \|\wt z \|_{L^6_{T, x}}^6
+ \int_0^t \|v (t')  \|_{L^6_{x}}^6 dt'
\les \|\wt z \|_{L^6_{T, x}}^6
+ \int_0^t E(v)(t') dt'.
\label{E4b}
\end{align}

\smallskip

\noi
$\bullet$ {\bf Case 2:} $M \geq 2$.

We write
\[ v^5 
= \sum_{\substack{M_j \geq 1\\\text{dyadic}}}\prod_{j = 1}^5\P_{M_j}  v ,\]

\noi
and assume that $M_1 \geq M_2 \geq \cdots \geq M_5$
without loss of generality.
Note that we have 
$\P_{M} [ v(t)^5] = 0$
unless  $M_1 \ges M$.
With $M_1\ges M$ and using H\"older's inequality, we have 
\begin{align}
 |\mathcal I (t)|
 & \les 
\sum_{k = -1}^1 \sum_{\substack{M \geq 2\\\text{ dyadic}}} \sum_{\substack{M_1, \dots, M_5\\M_1 \ges M}}
\big\|\jb{\nb}^{s-\delta} \P_{2^k M}  \wt z(t)\big\|_{L^\infty_x}
M_1^{1-s+\delta} \bigg\| \prod_{j = 1}^5 \P_{M_j}v (t) \bigg\|_{L^1_x}  \notag \\
\intertext{Summing over dyadic $M, M_1, \dots, M_5$ (with a slight loss in a power of $M_1$) and applying Bernstein's inequality
followed by Young's inequality,}
& \les 
\sup_{M_1, \dots, M_5}
\big\|\jb{\nb}^{s-\delta}  \wt z(t)\big\|_{L^\infty_x}
M_1^{1-s+\delta+} \bigg\| \P_{M_1} v(t) \prod_{j = 2}^5 \P_{M_j}v  (t)\bigg\|_{L^1_x}  \notag  \\
& \les 
\sup_{M_1, \dots, M_5}
\big\|\jb{\nb}^{s-\delta}   \wt z(t) \big\|_{L^\infty_x}\Bigg\{
 \big\| M_1^{1-s+\delta+} \P_{M_1} v(t)\big\|_{L^3_x}^3  
 + \bigg\|\prod_{j = 2}^5 \P_{M_j}v (t) \bigg\|_{L^\frac{3}{2}_x}^\frac{3}{2} \Bigg\} \notag  \\
& \les 
\sup_{M_1, \dots, M_5}\big\|\jb{\nb}^{s-\delta}   \wt z(t)\big\|_{L^\infty_x}\Big\{
M_1^{3(1-s+\delta+)} \big\|  \P_{M_1} v(t)\big\|_{L^3_x}^3  
 + \prod_{j = 2}^5 \|\P_{M_j}v (t) \|_{L^6_x}^\frac{3}{2} \Big\}  \notag \\
& \les 
\sup_{M_1}\big\|\jb{\nb}^{s-\delta}   \wt z(t)\big\|_{L^\infty_x}\Big\{
  M_1^{3(1-s+\delta+)} \| \P_{M_1} v(t)\|_{L^3_x}^3  
 +  \|v (t) \|_{L^6_x}^6 \Big\}  \notag \\
\intertext{By interpolating $L^3$ between $L^2$ and $L^6$ and then applying Young's inequality,}
& \les 
\sup_{M_1} \big\|\jb{\nb}^{s-\delta}  \wt z(t)\big\|_{L^\infty_x}\Big\{
 \|   M_1^{2(1-s+\delta+)}  \P_{M_1} v\|_{L^2_x}^\frac{3}{2}  \|  \P_{M_1} v\|_{L^6_x}^\frac{3}{2}  
 +    E(v) \Big\} \notag  \\
& \les 
\sup_{M_1}
\big\|\jb{\nb}^{s-\delta}   \wt z(t)\big\|_{L^\infty_x}\Big\{
 \|   M_1^{2(1-s+\delta+)}  \P_{M_1} v\|_{L^2_x}^2 
 +  \|\P_{M_1} v  \|_{L^6_x}^6 + E(v) \Big\} \notag  \\
& \les 
\big\|\jb{\nb}^{s-\delta} \wt  z(t)\big\|_{L^\infty_x}E(v),
\label{E4c}
\end{align}

\noi
where the last inequality follows from Bernstein's inequality
as long as $2(1-s+\delta+) \leq 1$, i.e. $s > \frac{1}{2}+\delta$.

Hence, from \eqref{E4a}, \eqref{E4b}, and \eqref{E4c},  we obtain
\begin{align}
|\I_2 (t)| \les 
\|\wt z \|_{L^6_{T, x}}^6
+ 
\left(1+\big\|\jb{\nb}^{s-\delta}  \wt z\big\|_{L^\infty_{T, x}}\right)
\int_0^t 
E(v)(t') dt'.
\label{E5}
\end{align}

\noi
By choosing sufficiently small $a > 0$, 
it follows from
\eqref{E1}, \eqref{E2}, \eqref{E3}, \eqref{E4}, and \eqref{E5}
that
\begin{align*}
 E(v)(t)
\leq C_1(z, \wt z , T) + C_2(z, \wt z , T) \int_0^t E(v)(t') dt',
\end{align*}

\noi
for $t \in [0, T]$, where 
$C_1(z, \wt z, T)$ and  $C_2(z,\wt z,  T)$ satisfy 
\begin{align*}
C_1 (z, \wt z , T) & \sim \|z\|^{10}_{L^{10}_{T, x}} + \|z\|_{L^\infty_TL^6_x}^6
+ \|\wt z \|_{L^6_{T, x}}^6
 \leq \ld (T, \eps, \|(u_0,u_1)\|_{\mathcal{H}^s(\R^3)}) < \infty, \\
C_2 (z, \wt z, T) & \sim 1 + \|z\|^2_{L^\infty_{T, x}}  
+ \big\|\jb{\nb}^{s-\delta}  \wt z\big\|_{L^\infty_{T, x}} \leq \ld (T, \eps, \|(u_0,u_1)\|_{\mathcal{H}^s(\R^3)})< \infty.
\end{align*}

\noi
Finally, the energy bound \eqref{Penergy2} follows from Gronwall's inequality.
\end{proof}

\section{Deterministic analysis of the perturbed NLW}
\label{SEC:perturb}

In this section, we discuss  the deterministic component 
of  the proof of Theorem \ref{THM:GWP}.
Given a deterministic real-valued function $f$,
we consider the Cauchy problem of
the following perturbed defocusing quintic NLW:
\begin{align}
\begin{cases}
\pa_{t}^2v-\Delta v+(v+f)^5=0\\
(v, \dt v)|_{t = t_0} = (v_0, v_1).
\end{cases}
\label{pNLW}
\end{align}

\noi
 In this section,  we prove long time existence of solutions to \eqref{pNLW}
under some appropriate assumptions on $f$.

First, we briefly discuss the local well-posedness of \eqref{pNLW}.
If one applies the Strichartz estimates (Lemma \ref{LEM:Strichartz})
and a simple fixed point argument to 
prove local well-posedness of  \eqref{pNLW}
in the energy space, 
then the time of local existence depends on 
the profile of the initial data.
This, however, can be upgraded to the following
``good'' local well-posedness result, where 
the  time of local existence is characterized
only in terms of 
the $\dot {\mathcal{H}}^1$-norm of the initial data $(v_0,v_1)$
and the size of the perturbation $f$.

\begin{lemma}[Proposition 4.3 in \cite{Poc}]
\label{LEM:pLWP}
Let $(v_0,v_1)\in\mathcal{\dot H}^1(\R^3)$. 
Then, there exists a function 
$\tau : [0, \infty) \times \R_+\times \R_+ \to \R_+$,
non-increasing in the first two arguments, such that 
if  $f$ satisfies the condition
\begin{equation}\label{f}
\|f\|_{L^5_tL^{10}_x([t_0,t_0+\tau_\ast])}\leq K\tau_\ast^\theta
\end{equation}

\noi
for some $K,\theta>0$ and $\tau_\ast\leq \tau=\tau\big(\|(v_0,v_1)\|_{\dot{\mathcal{H}}^1(\R^3)}, K,\theta\big)
\ll1$, 
then  
there exists  a unique solution $(v, \dt v) \in C([t_0,t_0+\tau_\ast]; \dot{\mathcal{H}}^1(\R^3))$
to \eqref{pNLW}.

 \end{lemma}

Lemma \ref{LEM:pLWP}
is the exact analogue on $\R^3$ of Proposition 4.3 in \cite{Poc}. 
Its proof is based on the global space-time bounds of solutions to the 
energy-critical defocusing NLW from \cite{Bahouri_Gerard, Tao}
and  a perturbation argument, in particular, the long time perturbation lemma
(Lemma 4.5 in \cite{Poc}). See \cite{Poc} for the details of the proof.

Given finite $T \gg 1$, 
our goal is to construct a solution to \eqref{pNLW} on $[0, T]$
for suitable $f$.
If  there is an a priori energy control:
\begin{equation}
\sup_{t \in [0, T]} \|(v(t),\dt v(t))\|_{\dot{\mathcal{H}}^1(\R^3)} < C(T) < \infty,
\label{Fenergy0}
\end{equation}

\noi
then
Lemma \ref{LEM:pLWP} allows us to construct
a solution $v$ to \eqref{pNLW} on $[0, T]$.
Indeed, the analogue of Lemma \ref{LEM:pLWP} on $\R^d$,
$d=4,5$, was the main part of the deterministic analysis
in \cite{Poc}.
It was then combined with 
the probabilistic a priori energy bound \eqref{PocEnergy} 
to prove almost sure global well-posedness of the defocusing energy-critical NLW on $\R^d$, $d = 4, 5$.

Our setting is slightly different;
we do not assume an  a priori energy control \eqref{Fenergy0}. 
Instead,  we assume a  {\it  uniform} a priori energy control (see \eqref{Fenergy} below) on smooth approximating solutions $v_N$ 
and construct a solution $v$ to \eqref{pNLW} on long time intervals (Proposition \ref{PROP:pLWP2} below).
In Section \ref{SEC:GWP}, 
we will combine this result with 
the  uniform probabilistic a priori energy bound on smooth approximating solutions 
(Proposition \ref{PROP:Penergy})
and prove almost sure global well-posedness of \eqref{NLW}
below the energy space.

Given $f \in L^5_{t, \text{loc}}L^{10}_x$, 
  let $f_N = \P_{\leq N} f $ for dyadic $N \geq 1$.
Consider the following perturbed NLW:
\begin{align}
\begin{cases}
\pa_{t}^2v_N-\Delta v_N+(v_N+f_N)^5=0\\
(v_N, \dt v_N)|_{t = 0} = (0, 0).
\end{cases}
\label{pNLW2}
\end{align}

\noi
The following proposition is the main result of this section.

\begin{proposition}
\label{PROP:pLWP2}
Let $f, f_N$, and $v_N$ be as above.
Given finite $T > 0$, assume that the following conditions hold:
\begin{itemize}
\item[\textup{(i)}]
There exist $K, \theta > 0$ such that  
\begin{equation}
\|f\|_{L^5_tL^{10}_x(I\times \R^3)}\leq K|I |^\theta
\label{f2}
\end{equation}

\noi
for any compact interval $I \subset [0, T]$.	

\item[\textup{(ii)}]
For each dyadic $N\geq 1$, 
a solution $v_N$  to \eqref{pNLW2} exists on $[0, T]$
and 
satisfies the following uniform a priori energy bound:
\begin{equation}
\sup_N \sup_{t \in [0, T]} \|(v_N(t),\dt v_N(t))\|_{{\mathcal{H}}^1(\R^3)} < C_0(T) < \infty.
\label{Fenergy}
\end{equation}

\noi
\item[\textup{(iii)}]
There exists $\al > 0$ such that 
\begin{equation}
 \|f - f_N \|_{L^5_T L^{10}_x} < C_1(T) N^{-\al}
\label{Fapprox}
\end{equation}

\noi 
for all dyadic $N \geq 1$.

\end{itemize}

\noi
Then, there exists a unique solution $(v, \dt v)\in C([0, T]; {\mathcal{H}}^1(\R^3))$
 to \eqref{pNLW}
with $(v, \dt v)|_{t = 0} = (0, 0)$, 
satisfying  
\begin{equation}
 \sup_{t \in [0, T]} \|(v(t),\dt v(t))\|_{{\mathcal{H}}^1(\R^3)} < 2C_0(T) < \infty.
\label{Venergy}
\end{equation}

\end{proposition}

\begin{proof}
Given $T>0$, 
fix 
\begin{equation}
\tau_0 := \tau (2C_0(T), K, \theta),
\label{Fenergy1a}
\end{equation}

\noi
where  $\tau$ and $C_0(T)$
are  as in Lemma \ref{LEM:pLWP}
and \eqref{Fenergy}, respectively. 
Fix $0<\tau_\ast\leq \tau_0$
and 
 divide the time interval $[0, T]$
into $O\big(\frac{T}{\tau_\ast}\big)$-many subintervals of length $\tau_\ast$
and denote them by 
\begin{equation*}
I_j := [ j \tau_\ast, (j+1) \tau_\ast]\cap [0, T],
\end{equation*}

\noi
 $j = 0, 1, \dots,  \big[ \frac{T}{\tau_\ast}\big]$.
The basic idea of the proof is to iteratively apply Lemma \ref{LEM:pLWP}
on each $I_j$, 
while controlling  the growth of the $\mathcal{\dot H}^1$-norm of $(v, \dt v)$ on 
$I_j$.
In the following, various constants depend
on $K$, $\theta$, and $\al$ in \eqref{f2} and \eqref{Fapprox},
but we suppress their dependence.

We start with a brief description of the properties of 
the solution $v_N$ to \eqref{pNLW2}.
By \eqref{f2} and \eqref{Fapprox},
we have 
\begin{equation}
\|f_N\|_{L^5_tL^{10}_x(I\times\R^3)}\leq K|I|^\theta + C_1(T)N^{-\alpha},
\label{f_N}
\end{equation}

\noi
for any compact interval $I\subset [0,T]$.
It follows from a slight modification of the 
 proof of Proposition 4.3 in \cite{Poc}
that there exists $N_1=N_1(T,\|(v_0,v_1)\|_{\dot{\mathcal{H}}^1(\R^3)})$ 
such that
an analogue of Lemma \ref{LEM:pLWP} holds
for 
\begin{align}
\begin{cases}
\pa_{t}^2v_N-\Delta v_N+(v_N+f_N)^5=0\\
(v_N, \dt v_N)|_{t = t_0} = (v_0, v_1)
\end{cases}
\label{XX1}
\end{align}

\noi
as long as $N\geq N_1$.
More precisely, there exists $\tau_1=\tau_1\big(\|(v_0,v_1)\|_{\dot{\mathcal{H}}^1(\R^3)}, K,\theta\big)
\ll1$
such that, 
if $f_N$ satisfies \eqref{f_N}
on $I = [t_0, t_0 + \wt\tau]$
for some $0<\wt \tau \leq \tau_1$, then
there exists a unique solution 
$(v_N, \dt v_N) \in C([t_0,t_0+ \wt \tau]; \dot{\mathcal{H}}^1(\R^3))$
to \eqref{XX1}.

Let $\tau_2:= \tau_1 (C_0(T),K,\theta )$.
Then, in view of \eqref{Fenergy},
we can apply this observation  iteratively on intervals 
$\tilde I_k:=[k\tau_2,(k+1)\tau_2]$, 
$k=0,1,\dots, \big[\frac{T}{\tau_2}\big]$,
and 
 define the solution $v_N$
on the whole interval $[0,T]$.
Moreover, 
it follows from
the proof of Proposition 4.3 in \cite{Poc}
that there exist $\eta\ll 1$ 
and $J(C_0(T)) \in \mathbb{N}$
such that 
we can decompose  the time interval interval $\tilde I_k$
into $J_k'$-many  subintervals $\tilde I_{k, \ell}$ for some $J'_k \leq J(C_0(T))$
with the property that 
\begin{equation}\label{FenergyX}
\|v_N\|_{L^5_{\tilde I_{k,\ell}}L^{10}_x}\leq 4\eta,
\end{equation}

\noi
for all $k = 0, 1, \dots, \big[\frac{T}{\tau_2}\big]$
and $j = 1, 2, \dots, J'_k$.

We now begin the construction of the solution $v$ to \eqref{pNLW}.
Since $(v,\partial_t v)|_{t=0}=(0,0)$, 
we  have $\| (v(0), \dt v(0))\|_{\dot{\mathcal  H}^1(\R^3)} = 0\leq 2C_0(T)$.
Thus, 
Lemma \ref{LEM:pLWP}
guarantees the existence of $v$ on $I_0:=[0,\tau_\ast] \subset [0, \tau_0]$.
In particular, we have $(v, \dt v)  \in C(I_0;\dot{\mathcal{H}}^1(\R^3))$. 
Moreover, 
it follows from
the proof of Proposition 4.3 in \cite{Poc}
that there exists
a decomposition of the time interval $I_0$
into $J_0$-many subintervals $I_{0, m}$,
with the property that 
\begin{equation}\label{Fenergy1}
\|v\|_{L^5_{ I_{0,m}}L^{10}_x}\leq 4\eta, 
\end{equation}

\noi
for all $m = 1, 2, \dots, J_0$.

Next,  consider the following decomposition 
of $I_0$:
\[I_0=\bigcup_{k,\ell, m} \big\{I_{0,k,\ell, m}:=I_{0,m}\cap \tilde{I}_{k,\ell} 
\, : \,  \tilde{I}_{k,\ell}\cap I_0\neq \emptyset \big\}. \]

\noi
Note that this decomposition
contains at most $\big(\big[\frac{T}{\tau_2}\big]+1\big)J(C_0(T))J_0$ subintervals.
For notational simplicity, 
let 
$I := I_{0,k, \l, m}$, 
$t_0 := \min I_{0,k, \l, m}$,
and 
$w_N : = (v - v_N, \dt v - \dt v_N)$.
Then, 
it follows 
from
 \eqref{Fenergy1}, \eqref{FenergyX},
\eqref{f2}, \eqref{f_N},
and making $\tau_0$ smaller, if necessary, 
that there exists $N_2 = N_2(T) \geq N_1$ such that 
\begin{align}
 \| v\|_{L^5_I L^{10}_x}^4  + \| v_N\|_{L^5_I L^{10}_x}^4
& + \| f\|_{L^5_I L^{10}_x}^4 + \| f_N\|_{L^5_I L^{10}_x}^4 \notag \\
& \lesssim \eta^4+K^4\tau_0^{4\theta} +[C_1(T)N^{-\alpha}]^4 \ll 1,
\label{Fenergy3}
\end{align}
	
\noi
for all $N \geq N_2$.
Then,  by Lemma \ref{LEM:Strichartz}
and \eqref{Fenergy3}, we have
\begin{align}
\|w_N \|_{L^\infty_I \dot{\mathcal{H}}^1} 
& +  \| v - v_N\|_{L^5_I L^{10}_x}\notag \\
& \leq C_2 \|w_N(t_0)\|_{ \dot{\mathcal{H}}^1} 
+\tfrac 12 \| v - v_N\|_{L^5_I L^{10}_x}
 + 
\tfrac 12  \| f - f_N\|_{L^5_I L^{10}_x}.
\label{Fenergy4}
\end{align}

\noi
Hence, 
it follows from  \eqref{Fapprox} and \eqref{Fenergy4} that
\begin{align}
\|w_N \|_{L^\infty_I \dot{\mathcal{H}}^1} 
 +  \| v - v_N\|_{L^5_I L^{10}_x}
 \leq C_3 (T) \big( \|w_N(t_0)\|_{ \dot{\mathcal{H}}^1} 
 + 
N^{-\al}\big)
\label{Fenergy5}
\end{align}

\noi
for all $N \geq N_2$.
Then, 
applying \eqref{Fenergy5} 
with $w_N(0) = 0$  and \eqref{L2}
on all the subintervals $I = I_{0,k, \l, m}$
in an iterative manner, 
we obtain
\begin{align}
\| w_N \|_{L^\infty_{I_0} {\mathcal{H}}^1} 
 \leq  T(C_3 (T) +1)^{([\frac{ T}{\tau_2}]+1)J(C_0(T))J_0}N^{-\al}.
\label{Fenergy5a}
\end{align}

\noi
Then, it follows from 
\eqref{Fenergy5a}   and \eqref{Fenergy} that
there exists $N_3 = N_3(T, \tau_2) \geq N_2 $ such that 
\begin{align}
\| (v, \dt v) \|_{L^\infty_{I_0} {\mathcal{H}}^1} 
 \leq C_0(T) + T(C_3 (T) +1)^{([\frac{ T}{\tau_2}]+1)J(C_0(T))J_0}N^{-\al}
\leq 2 C_0(T)
\label{Fenergy6}
\end{align}

\noi
for all $N \geq N_3$.
This in particular implies that 
\[\|(v(\tau_\ast),\dt v(\tau_\ast))\|_{\dot {\mathcal{H}}^1(\R^3)}  \leq 2C_0(T).\]

\noi
Thus, we can apply Lemma \ref{LEM:pLWP}
and construct a solution $(v, \dt v)  \in C(I_1;\dot{\mathcal{H}}^1)$.
Moreover, 
it follows from
the proof of Proposition 4.3 in \cite{Poc}
that there exists
a decomposition of the time interval $I_1$
into $J(2C_0(T))$-many subintervals $I_{1, m}$
with the property that 
\begin{equation*}
\|v\|_{L^5_{ I_{1,m}}L^{10}_x}\leq 4\eta, 
\end{equation*}

\noi
for all $m = 1, 2, \dots, J(2C_0(T))$.
Arguing as before,
 there exists $N_4 = N_4(T, \tau_2) \geq N_3 $ such that 
\begin{align}
\| (v, \dt v) \|_{L^\infty_{I_1} {\mathcal{H}}^1} 
&  \leq C_0(T) + T(C_3 (T) +1)^{([\frac{ T}{\tau_2}]+1)J(C_0(T))(J_0+J(2C_0(T))}N^{-\al} \notag\\
& \leq 2 C_0(T)
\label{Fenergy7}
\end{align}
for all $N\geq N_4$.
In view of \eqref{Fenergy7}, 
we can clearly apply 
Lemma \ref{LEM:pLWP} 
and extend the solution $v$ onto $I_2$.

Arguing inductively, 
we  can extend the solution $v$ onto the entire interval $[0, T]$. 
Furthermore, there exists $N_0=N_0(T,\tau_2, \tau_\ast)\in \mathbb{N}$ such that 
\begin{align*}
\sup_{t \in [0, T]}\| (v(t), \dt v(t)) \|_{{\mathcal{H}}^1(\R^3)} 
& \leq C_0(T) + T(C_3 (T) +1)^{([\frac{ T}{\tau_2}]+1)J(C_0(T))
\{ J_0 + [\frac{T}{\tau_\ast}]J(2C_0(T))\}}N^{-\al}\notag\\
&< 2 C_0(T),
\end{align*}
for all $N\geq N_0$.
Hence, the energy estimate \eqref{Venergy} is also satisfied on $[0,T]$.
\end{proof}

\begin{remark}\label{REM:pLWP}\rm
(i) 
The condition 
 \eqref{Fenergy} can be relaxed as follows; 
  it suffices to assume 
\begin{equation*}
 \sup_{t \in [0, T]} \|(v_{N_0}(t),\dt v_{N_0}(t))\|_{{\mathcal{H}}^1(\R^3)} < C_0(T) < \infty,
\end{equation*}

\noi
for some $N_0 = N_0(T, C_0(T)) \gg 1$.

\smallskip

\noi
(ii) 
The proof of Proposition \ref{PROP:pLWP2} 
shows that the hypothesis \eqref{f2} can also be relaxed. 
Let $\tau_* \leq \tau_0$, where $\tau_0$ is as in \eqref{Fenergy1a}.
Then, by setting $ I_j = [ j \tau_*, (j+1) \tau_*]\cap [0, T]$, it suffices to assume 
that there exist $K, \theta > 0$ such that  
\begin{equation*}
\|f\|_{L^5_{I_j} L^{10}_x}\leq K| I_j |^\theta \ll 1
\end{equation*}

\noi
for  all $j = 0, \dots, \big[\frac{T}{\tau_*}\big]$.

\end{remark}



\section{Almost sure global existence}
\label{SEC:GWP}

In this section, we present the proof of  Theorem \ref{THM:GWP}.
Note that 
Theorem \ref{THM:GWP}
follows once we prove  the following  `almost' almost sure global well-posedness for \eqref{NLW}.
See \cite{Colliand_Oh, BOP2} for details on this reduction.

\begin{proposition}[`Almost' almost sure global well-posedness]
\label{PROP:aas}
Let $s \in (\frac 12,  1)$ and $T\geq 1$.
Given  $(u_0, u_1)\in \mathcal{H}^{s}(\R^3)$,
let 
$(u_0^\omega, u_1^\omega)$ be the Wiener randomization
defined in \eqref{R1}, satisfying \eqref{cond}. 
Then,
given any $T, \eps > 0$,  there exists $\Omega_{T,\eps}\subset \Omega$
such that 
\begin{itemize}
\item[\textup{(i)}]
 $P(\Omega_{T,\eps}^c)<\eps$, 
\item[\textup{(ii)}]
For any $\omega\in \Omega_{T,\eps}$, there exists a unique solution $u^\omega$ to \eqref{NLW} on $[0,T]$
with  $(u^\omega, \pa_t u^\omega)|_{t=0}=(u_0^\omega, u_1^\omega)$ in the class:
\begin{align*}
\big(S(t)(u_0^\omega, u_1^\omega), \dt S(t)(u_0^\omega, u_1^\omega)\big)
+C([0, T]; \mathcal{H}^1(\R^3))
 \subset C([0, T]; \mathcal{H}^s(\R^3)).
\end{align*}
\item[\textup{(iii)}] For any $\omega\in \Omega_{T,\eps}$,  
the following probabilistic energy bound holds for the nonlinear part $v^\omega$ of 
the solution $u^\omega$:
\begin{equation*}
 \sup_{t \in [0, T]} \|(v^\omega(t),\dt v^\omega(t))\|_{{\mathcal{H}}^1(\R^3)} < C(T,\eps,\|(u_0,u_1)\|_{\mathcal{H}^s(\R^3)}).
\end{equation*}
\end{itemize}

\end{proposition}


The main ingredients of the proof of Proposition \ref{PROP:aas} are 
the probabilistic uniform energy bound on approximating solutions
(Proposition \ref{PROP:Penergy})
and 
the deterministic long time existence for the perturbed NLW \eqref{pNLW}
(Proposition \ref{PROP:pLWP2}).


\begin{proof}

\noi
Given $(u_0^\omega, u_1^\omega)$, 
let $z^\o$ and $z_N^\o$ be as in \eqref{z} and \eqref{Penergy0}, respectively.
With $\al \in (0, s]$, 
set 
\[ M  = M(T, \eps,  \| (u_0, u_1)\|_{\mathcal{H}^{\al}} 
)\sim  T^\frac{6}{5} \Big(\log \frac{1}{\eps} \Big)^\frac{1}{2}
 \| (u_0, u_1)\|_{\mathcal{H}^{\al} }.\]

\noi
Then, 
defining $\O_1 =\O_1 (T, \eps)$
by 
\[ \O_1 := \{\o\in \O: \,  \| \jb{\nb}^\al z^\o  \|_{L^5_T L^{10}_x}
\leq M \}, \]

\noi
it follows from 
Lemma \ref{LEM:Str} (ii) that 
\begin{equation}
 P(\O_1^c) < \frac{\eps}{3}.
\label{Q1} 
 \end{equation}

\noi
Moreover, for each $\o \in \O_1$, 
we have 
\begin{align}
\| z^\o
-  z_N^\o \|_{L^5_T L^{10}_x}
\leq N^{-\al} 
\| \jb{\nb}^\al z^\o  \|_{L^5_T L^{10}_x}
\leq M N^{-\al}.
\label{Q2} 
\end{align}

Given dyadic $N\geq 1$, 
apply Proposition \ref{PROP:Penergy}
and 
construct $\O_2 (N) := \wt{\Omega}_{N, T,\frac{\eps}{3}}$ with 
\begin{align}
P(\Omega_2(N)^{c})<\frac{\eps}{3}
\label{Q3} 
\end{align}

\noi
such that 
\begin{align}
\sup_{t \in [0, T]}\| (v^\o_N  (t), \dt v^\o_N(t)) \|_{ \mathcal{H}^1} 
\leq C_0(T, \eps,  \|(u_0, u_1)\|_{{\mathcal{H}}^s}) < \infty,
\label{Q4}
\end{align}

\noi
 for each $\o \in \O_2(N)$.
The main point here is that 
 $C_0 = C_0(T, \eps,  \|(u_0, u_1)\|_{{\mathcal{H}}^s})$ 
 can be chosen independent of $N$.

Fix  $K =  \| (u_0, u_1)\|_{\mathcal{H}^{0}}$
and $\theta = \frac{1}{10}$ in the following.
Let $\tau_* \leq \tau_0$ 
to be chosen later, 
where  $\tau_0 = \tau\big(2C_0(T), K, \theta)$ is as in \eqref{Fenergy1a}. 
By writing 
 $[0, T] = \bigcup_{j = 0}^{[ T/\tau_*]}  I_j$
with   $ I_j = [ j \tau_*, (j+1) \tau_*]\cap [0, T]$, 
define $ \O_3$ by 
\begin{align}
\O_3 := \Big\{ \o \in \O: \, 
\| z^\o \|_{L^5_{ I_j} L^{10}_x} \leq K | I_j|^{\theta}, 
j = 0, \dots, \big[\tfrac{T}{\tau_*}\big]\Big\}.
\label{Q4a}
\end{align}

\noi
Then, by Lemma \ref{LEM:Str} with $|I_j | \leq \tau_*$, 
we have 
\begin{align*}
P(\O_3^c) 
& \leq 
\sum_{j = 0}^{[\frac{T}{\tau_*}]}
P \Big( \| z^\o \|_{L^5_{I_j} L^{10}_x} > K |I_j |^{\theta}\Big)
 \les \frac{T}{\tau_*} \exp \Bigg( - \frac{c}{T^2 \tau_*^\frac{1}{5}}\Bigg).
\intertext{By making $\tau_*$ smaller if necessary,}
& \les \frac{T}{\tau_*} \tau_* \exp \Bigg( - \frac{c}{2T^2 \tau_*^\frac{1}{5}}\Bigg)
 =  T \exp \Bigg( - \frac{c}{2T^2 \tau_*^\frac{1}{5}}\Bigg).
\end{align*}

\noi
Hence, by choosing $\tau_* = \tau_*(T, \eps)$ 
sufficiently small, 
we have 
\begin{align}
P(\Omega_3^{c})<\frac{\eps}{3}.
\label{Q5} 
\end{align}

Let $\O_{T, \eps}: = \O_1 \cap \O_2(N_0) \cap \O_3$, 
where $N_0$ is to be chosen later.
Then, from \eqref{Q1}, \eqref{Q3}, and \eqref{Q5}, we have 
\begin{align*}
P(\Omega_{T, \eps}^{c})<\eps.
\end{align*}

\noi
By choosing $N_0 = N_0  (T, \eps,  \|(u_0, u_1)\|_{{\mathcal{H}}^s}) \gg 1$, 
it follows from 
Proposition \ref{PROP:pLWP2}
and Remark \ref{REM:pLWP}
with \eqref{Q2}, 
\eqref{Q4},
and  
\eqref{Q4a},
that 
there exists a solution $v^\o$ to \eqref{v} on $[0, T]$ for each $\o \in \O_{T, \eps}$.
Hence, for $\o \in \O_{T, \eps}$, 
there exists a solution $u^\o = z^\o + v^\o$ to \eqref{NLW} on $[0, T]$.
Moreover, the following estimate holds:
\begin{equation*}
 \sup_{t \in [0, T]} \|(v^\omega(t),\dt v^\omega(t))\|_{{\mathcal{H}}^1(\R^3)} < 2C_0(T,\eps,\|(u_0,u_1)\|_{\mathcal{H}^s(\R^3)})<\infty.
\end{equation*}
This completes the proof of Proposition \ref{PROP:aas}
and hence the proof of Theorem \ref{THM:GWP}.
\end{proof}

\begin{ackno}\rm
T.O.~was supported by the European Research Council (grant no.~637995 ``ProbDynDispEq'').

\end{ackno}



\end{document}